\documentclass[10pt,a4paper,reqno,draft]{amsart}
\title[Large data global solutions for the cubic Dirac equation]%
{On the cubic Dirac equation with potential
and the Lochak--Majorana condition}
\usepackage{%
    amssymb%
    ,graphicx%
    ,mathrsfs%
    ,eucal%
    ,xcolor%
    ,enumitem%
    ,geometry%
    ,esint}
\usepackage[%
    pdfauthor={Piero D'Ancona}%
    ,colorlinks=true%
    ,linkcolor=magenta%
    ,citecolor=cyan%
    ,urlcolor=blue%
    ,hyperfootnotes=false%
]{hyperref}

\numberwithin{equation}{section}
\newtheorem{theorem}{Theorem}[section]
\newtheorem{corollary}[theorem]{Corollary}
\newtheorem{lemma}[theorem]{Lemma}

\newtheorem{proposition}[theorem]{Proposition}
\theoremstyle{remark}
\newtheorem{remark}[theorem]{Remark}

\theoremstyle{definition}
\newtheorem{definition}[theorem]{Definition}
\usepackage{stmaryrd}
  {\baselineskip0pt\normalfont\footnotesize
  \setlength\abovedisplayskip{1pt}
  \setlength\abovedisplayshortskip{1pt}
  \setlength\belowdisplayskip{1pt}
  \setlength\belowdisplayshortskip{1pt}
  \vskip2pt\par\noindent$\llbracket$\ignorespaces}%
  {\ignorespaces$\rrbracket$\vskip2pt}

\newcommand{\bra}[1]{\langle #1 \rangle}
\newcommand{\one}[1]{\mathbf{1}_{#1}}

\newcommand{\R}{\mathbb{R}}
\newcommand{\C}{\mathbb{C}}
\newcommand{\D}{\mathcal{D}}
\newcommand{\Si}{\mathbb{S}}

\newcommand{\eps}{\epsilon}

\allowdisplaybreaks[1]

\date{\today}

\author[P.~D'Ancona]{Piero D'Ancona}
\address{Piero D'Ancona: 
Dipartimento di Matematica\\
Sapienza Universit\`{a} di Roma\\
Piazzale A.~Moro 2\\
00185 Roma\\
Italy}
\email{dancona@mat.uniroma1.it}
\author[M.~Okamoto]{Mamoru Okamoto}
\address{Mamoru Okamoto:
Division of Mathematics and Physics\\
Faculty of Engineering\\
Shinshu University\\
4-17-1 Wakasato\\
Nagano City 380-8553\\
Japan}
\email{m\_okamoto@shinshu-u.ac.jp}
\thanks{%
}

\begin{document}

\begin{abstract}
  We study a cubic Dirac equation on $\mathbb{R}\times\mathbb{R}^{3}$
  \begin{equation*}
    i \partial _t u + \D u  + V(x) u = 
    \langle \beta u,u \rangle \beta u
  \end{equation*}
  perturbed by a large potential with almost critical regularity. We prove global existence and scattering for small initial data in $H^{1}$ with additional angular regularity. The main tool is an endpoint Strichartz estimate for the perturbed Dirac flow. In particular, the result covers the case of spherically symmetric data with small $H^{1}$ norm.
  
  When the potential $V$ has a suitable structure, we prove global existence and scattering for \emph{large} initial data having a small chiral component, related to the Lochak--Majorana condition.
\end{abstract}

\maketitle

%%% >>> INDEX (toc)
%\tableofcontents

% e_f_pre  <<<<<<<<< PREAMBOLO

\section{Introduction}

We consider the Cauchy problem for a 
cubic Dirac equation with potential
\begin{equation} \label{eq:Dirac}
  i \partial _t u + \D u  + V(x) u = 
  \langle \beta u,u \rangle \beta u, \qquad
  u(0,x) = u_0(x).
\end{equation}
in an unknown function
$u = u(t,x) : \R \times \R ^3 \rightarrow \C ^4$,
with initial data $u_{0}:\mathbb{R}^{3}\to \mathbb{C}^{4}$.
Here $\langle\cdot , \cdot\rangle$ is the $\C^4$ inner product, 
$\mathcal{D}$ is the Dirac operator defined by
\begin{equation*}
  \textstyle
  \mathcal{D}=i^{-1}\sum _{j=1}^3 \alpha _j \partial _j
  =
  i^{-1}\alpha \cdot \partial,
  \qquad
  \alpha=(\alpha_{1},\alpha_{2},\alpha_{3})
\end{equation*}
where $\partial=(\partial_{1},\partial_{2},\partial_{3})$
are the partial derivatives,
and $\beta$, $\alpha _j$ are the Dirac matrices
\begin{equation*}
  \scriptstyle
  \beta=
  \begin{pmatrix} 1 & 0 & 0 & 0 \\ 0 & 1 & 0 & 0 \\ 0 & 0 & -1 & 0 \\ 0 & 0 & 0 & -1 \end{pmatrix},
  \ 
  \alpha_1 = \begin{pmatrix} 0 & 0 & 0 & 1 \\ 0 & 0 & 1 & 0 \\ 0 & 1 & 0 & 0 \\ 1 & 0 & 0 & 0 \end{pmatrix},
  \ 
  \alpha_2 = \begin{pmatrix} 0 & 0 & 0 & -i \\ 0 & 0 & i & 0 \\ 0 & -i & 0 & 0 \\ i & 0 & 0 & 0 \end{pmatrix},
  \ 
  \alpha_3 = \begin{pmatrix} 0 & 0 & 1 & 0 \\ 0 & 0 & 0 & -1 \\ 1 & 0 & 0 & 0 \\ 0 & -1 & 0 & 0 \end{pmatrix}.
\end{equation*}
We recall the basic anticommuting relations
\begin{gather*}
  \alpha _j ^{\ast}  = \alpha _j,\quad
  \beta ^{\ast}  = \beta, \quad
  \beta ^2 = I_4, \quad
  \beta \alpha _j + \alpha _j \beta = 0 
  \quad\text{for $j =1, 2, 3$},
  \\
  \alpha _j \alpha _k + \alpha _k \alpha _j = 2 \delta _{jk} I_4 
  \quad\text{for $j,k =1, 2 , 3$},
\end{gather*}
where $M^{*}$ is the conjugate transpose of the matrix $M$,
$\delta _{jk}$ the Kronecker delta and $I_4$ the $4 \times 4$
identity matrix.

Concerning the potential $V(x):\mathbb{R}^{3}\to M_{4}(\mathbb{C})$,
we decompose it in the form
\begin{equation}\label{eq:decompV}
  V=\sum_{j=1}^{3}A_{j}(x)\alpha_{j}+A_{0}(x)\beta +V_{0}(x)=
  A \cdot \alpha +A_{0}\beta+V_{0}
\end{equation}
where the \emph{magnetic potential} $A$, the 
\emph{pseudoscalar potential} $A_{0}$ and $V_{0}$ are such that
\begin{equation*}
  A=(A_{1},A_{2},A_{3}):\mathbb{R}^{3}\to \mathbb{R}^{3},
  \qquad
  A_{0}:\mathbb{R}^{3}\to \mathbb{R},
  \qquad
  V_{0}=V_{0}^{*}:\mathbb{R}^{3}\to M_{4}(\mathbb{C}).
\end{equation*}
The \emph{magnetic field} associated to the
potential $A$ will be denoted by
\begin{equation}\label{eq:magf}
  B=[B_{jk}]_{j,k=1}^{3},
  \qquad
  B_{jk}=\partial_{j}A_{k}-\partial_{k}A_{j},
  \qquad
  j,k=1,2,3.
\end{equation}

The first goal of the paper is to study the dispersive
properties of the Dirac flow perturbed by a \emph{large} 
potential, and to prove several
smoothing and (endpoint) Strichartz estimates for it.
We then apply the estimates to prove
the global existence of small solutions for the nonlinear
equation \eqref{eq:Dirac},
for $H^{1}$ initial data with additional
angular regularity, in the spirit of \cite{MNNO} and 
\cite{CacDAn}. Moreover, if the potential has an additional
structure, we are able to reduce the smallness assumption to
smallness of the chiral component of the initial data;
to this end we exploit the Lochak--Majorana condition.

A crucial but natural assumption concerns the absence of a
resonance at 0 for the operator $\mathcal{D}+V$.
It is well known that in presence of
a resonance the dispersive proerties of the flow 
deteriorate. For the Dirac equation with potential, the natural
notion is the following:

\begin{definition}[Resonance at 0]\label{def:resonDirac}
  We say that $0$ is a \emph{resonance} for the operator
  $\mathcal{D}+V$ if there exists 
  $v\in H^{1}_{loc}(\mathbb{R}^{3}\setminus0)$
  solution of $(\mathcal{D}+V) v=0$
  such that $|x|^{-\frac 12-\sigma}v\in L^{2}$
  for all $\sigma$ in a right nbd of 0;
  $v$ is called a \emph{resonant state}.
\end{definition}

In order to state the results we introduce the dyadic norms
\begin{equation}\label{eq:dyadicsp}
  \|v\|_{\ell^{p}L^{q}}
  :=
  \Bigl(
  \sum_{j\in \mathbb{Z}}
    \|v\|_{L^{q}(2^{j}\le|x|<2^{j+1})}^{p}
  \Bigr)^{1/p},
\end{equation}
with obvious modification when $p=\infty$.
More generally, we denote the
mixed radial--angular $L^{q}L^{r}$ norms on a
spherical ring $C= \{R_{1}\le|x|\le R_{2} \}$ with
\begin{equation*}
  \|v\|_{L^{q}_{|x|}L^{r}_{\omega}(C)}=
  \|v\|_{L^{q}L^{r}(C)}:=
  \textstyle
  (\int_{R_{1}}^{R_{2}}
    (\int_{|x|=\rho}|v|^{r}dS)^{q/r}d\rho)^{1/q}.
\end{equation*}
and we define for all $p,q,r\in[1,\infty]$
\begin{equation}\label{eq:lpLp}
  \|v\|_{\ell^{p}L^{q}L^{r}}
  :=
  \|\{\|v\|_{L^{q}L^{r}(2^{j}\le|x|<2^{j+1})}\}
  _{j\in \mathbb{Z}}\|_{\ell^{p}}.
\end{equation}
Clearly, when $q=r$ we have simply
$ \|v\|_{\ell^{p}L^{q}L^{q}}=\|v\|_{\ell^{p}L^{q}}$.
In the following we shall also need mixed space--time norms,
and to avoid confusion we shall always write $L^{p}_{t}$
with an explicit index $t$,
the $L^{p}$ norms with respect to time variable.
Thus we write
\begin{equation*}
  \textstyle
  \|u\|_{L^{p}_{t}L^{q}L^{r}}
  :=
  \|u\|_{L^{p}_{t}L^{q}_{|x|}L^{r}_{\omega}}
  =
  \Bigl(
    \int\|u(t,x)\|_{L^{q}_{|x|}L^{r}_{\omega}}^{p}
    dt
  \Bigr)^{1/p}
\end{equation*}
Angular regularity will be expressed via fractional powers of
the Laplace--Beltrami operator on the sphere $\mathbb{S}^{2}$
\begin{equation*}
  \Lambda^s_{\omega} := (1-\Delta_{\Si^2})^{\frac{s}{2}}.
\end{equation*}

We shall impose several
decay and smoothness conditions on the potential $V$.
The minimal set of assumptions is the following.

\medskip
\textbf{Condition (V)}.
The operator $\mathcal{D}+V$ 
is selfadjoint on $L^{2}(\mathbb{R}^{3};\mathbb{C}^{4})$
with domain $H^{1}(\mathbb{R}^{3})$,
0 is not an resonance,
and $V$ satisfies (see \eqref{eq:decompV}, \eqref{eq:magf})
$A\in \ell^{\infty}L^{3}$, and
\begin{equation}\label{eq:assIIdirac}
  \||x|V_{0}\|_{\ell^{1}L^{\infty}}< \sigma
\end{equation}
\begin{equation}\label{eq:assIdirac}
  |V|^{2}+|\mathcal{D}V|\lesssim|x|^{-2-\delta},
  \quad
  |x||B|+
  |x|^{2}(|V|^{2}+|\mathcal{D}V|+|\mathcal{D}V_{0}|)
  \in \ell^{1}L^{\infty},
\end{equation}
for some $\delta>0$
(recall \eqref{eq:magf}).
\smallskip

The decay properties of the flow $e^{it(\mathcal{D}+V)}$
are summarized in Theorem \ref{the:smooDV} below.
For the following statement, we fix a 
radially symmetric weight function
$\rho\in \ell^{2}L^{\infty}$ such that $\rho^{-2}|x|$
is in the Muckenhoupt class $A_{2}$ on $\mathbb{R}^{3}$;
possible explicit examples for $\rho$ are 
\begin{equation*}
  \rho=|x|^{\epsilon}
  \ \text{for}\ |x|\le1
  \qquad
  \rho=|x|^{-\epsilon}
  \ \text{for}\ |x|\ge1
  \qquad
\end{equation*}
for some $e>0$ small, or also
\begin{equation*}
  \rho=\bra{\log|x|}^{-\nu}
\end{equation*}
for some $\nu>1/2$.
Recall that a locally integrable function $w>0$ is in $A_{2}$
if its averages over arbitrary balls $B$ satisfy 
$\fint_{B}w \cdot\fint_{B} w^{-1}\lesssim1$.

\begin{theorem}[Linear decay estimates]\label{the:smooDV}
  \leavevmode
  
  \noindent
  (i) (Smoothing estimate) 
  If Condition (V) holds with $\sigma$ small enough, then
  \begin{equation}\label{eq:smooDiracV0}
    \|\rho|x|^{-1/2} e^{it(\mathcal{D}+V)}f\|_{L^{2}_{t}L^{2}}
    \lesssim
    \|f\|_{L^{2}}.
  \end{equation}
  % \begin{equation}\label{eq:smooDiracVnh0}
  %   \textstyle
  %   \||x|^{-1/2}\rho\int_{0}^{t}e^{i(t-t')(\mathcal{D}+V)}F(t')dt'
  %   \|_{L^{2}_{t}L^{2}}
  %   \lesssim
  %   \||x|^{1/2}\rho^{-1}F\|_{L^{2}_{t}L^{2}}.
  % \end{equation}
  % If in addition for some $\rho\in \ell^{2}L^{\infty}$
  % one has
  % $\rho^{-2}|x|^{2}|\partial V|\in L^{\infty}$,
  % then
  % \begin{equation}\label{eq:smooDiracVD0}
  %   \||x|^{-1/2}\rho\ \partial e^{it(\mathcal{D}+V)}f\|
  %   _{L^{2}_{t}L^{2}}
  %   \lesssim
  %   \|f\|_{H^{1}}.
  % \end{equation}

  \noindent
  (ii) (Endpoint Strichartz estimate)
  If Condition (V) holds with $\sigma$ small enough, 
  and in addition
  $\rho^{-2}|x|(|V|+|\partial V|)\in L^{\infty}$,
  then
  \begin{equation}\label{eq:strichDV0}
    \|
    e^{it(\mathcal{D}+V)}f\|_{L^{2}_{t}L^{\infty}L^{2}}
    \lesssim
    \|
    f\|_{H^{1}}.
  \end{equation}

  \noindent
  (iii) (Estimates with angular regularity)
  Let $V$ be of the form $V=A_{0}\beta+V_{0}$,
  satisfying Condition (V).
  Assume in addition that for some $1<s\le2$
  \begin{equation}\label{eq:assVang}
    \rho^{-2}|x|\|\Lambda^{s}V_{0}(|x| \omega)\|
    _{L^{2}_{\omega}(\mathbb{S}^{2})}\le \sigma,
    \qquad
    \rho^{-2}|x|
    \|\Lambda^{s}\partial V(|x| \omega)\|
    _{L^{2}_{\omega}(\mathbb{S}^{2})}
    \in L^{\infty},
  \end{equation}
  \begin{equation}\label{eq:assVangA0}
    \rho^{-2}|x|
    \|x\wedge \partial A_{0}(|x| \omega)\|
    _{L^{\infty}_{\omega}(\mathbb{S}^{2})}
    +
    \rho^{-2}\bra{x}
    \|\Delta_{\mathbb{S}} A_{0}(|x| \omega)\|
    _{L^{\infty}_{\omega}(\mathbb{S}^{2})}    
    \in L^{\infty}
  \end{equation}
  with $\sigma$ small enough.
  Then we have
  \begin{equation}\label{eq:strichDVla0}
    \|
    \Lambda^{s}_{\omega}
    e^{it(\mathcal{D}+V)}f\|_{L^{2}_{t}L^{\infty}L^{2}}
    +
    \|
    \Lambda^{s}_{\omega}
    e^{it(\mathcal{D}+V)}f\|_{L^{2}_{t}H^{1}}
    \lesssim
    \|
    \Lambda^{s}_{\omega}
    f\|_{H^{1}},
  \end{equation}
\end{theorem}

Note that the estimates in (iii) require smallness of the 
magnetic potential $A$ (so that $A$ can be absorbed in $V_{0}$).
On the other hand, the pseudoscalar potential $A_{0}$
can still be large.
Note also that \eqref{eq:assVangA0}
is trivially satisfied if $A_{0}$ is a radially symmetric
function.
The estimates are proved, besides several others,
in Theorem \ref{the:smoooestD},
Corollary \ref{cor:smooang} and
Theorem \ref{the:strichDV}
of Sections \ref{sec:smoodirac}--\ref{sec:str_est}.

As an application of the previous estimates, we
prove the global existence and scattering for initial
data small in the $\Lambda^{-s}_{\omega}H^{1}$ norm.
In particular, the result applies to all spherically
symmetric data with small $H^{1}$ norm.
For simplicity we restrict ourselves to the standard 
nonlinearity \eqref{eq:Dirac}, but it is clear that 
the same proof applies to more general cubic nonlinearitiea.

\begin{theorem}[Global existence, small data]\label{the:GWPsmall}
  If $V=A_{0}\beta+V_{0}$ satisfies the
  assumptions of Theorem \ref{the:smooDV}--(iii), then
  there exists $\epsilon_{0}>0$ such that, for any initial
  data $u_{0}$ with 
  $\|\Lambda^{s}_{\omega}u_{0}\|_{H^{1}}\le \epsilon_{0}$,
  Problem \eqref{eq:Dirac} has a unique global solution
  $u\in C H^{1}\cap L^{2}L^{\infty}$ with
  $\Lambda_{\omega}^{s}u\in L^{\infty} H^{1}$.
  Moreover $u$ scatters to a free solution, i.e.,
  there exists $u_{+}\in \Lambda^{-s}_{\omega}H^{1}$
  such that
  \begin{equation*}
    \lim_{t \rightarrow \infty} 
    \| \Lambda^s_{\omega} u(t) - \Lambda^s_{\omega} e^{it(\D + V)} u_+ \|_{H^1} =0.
  \end{equation*}
  A similar result holds for $t\to-\infty$.
\end{theorem}

In our last result we construct a family of large
global solutions to Equation \eqref{eq:Dirac},
related to the so called \emph{Lochak--Majorana condition}
(see \cite{Loc}, \cite{Bac}).
To define the condition we introduce the subspace $E$ of 
$\mathbb{C}^{4}$ defined by
\begin{equation}\label{eq:defE}
  E:=
  \{z\in \mathbb{C}^{4}
  \colon
  z_{1}=\overline{z}_{4},
  z_{2}=-\overline{z}_{3}
  \}
  =
  \{z\in \mathbb{C}^{4}
  \colon
  \gamma z=\overline{z}
  \},
\end{equation}
where $\gamma$ is the matrix
\begin{equation}\label{eq:defgamma}
  \gamma:= \begin{pmatrix} 0 & 0 & 0 & 1 \\ 0 & 0 & -1 & 0 \\ 0 & -1 & 0 & 0 \\ 1 & 0 & 0 & 0 \end{pmatrix}.
\end{equation}
Then we have:

\begin{definition}[LM condition]\label{def:LM}
  We say that a function 
  $f(x)\in L^{2}(\mathbb{R}^{3};\mathbb{C}^{4})$
  satisfies the \emph{Lochak--Majorana condition} if
  \begin{equation*}
    f(x)\in E
    \quad\text{for a.e.}\quad x
  \end{equation*}
  (or more generally, if $\exists \theta\in \mathbb{R}$
  such that $e^{i \theta}f\in E$ for a.e.~$x$.)
\end{definition}
A few elementary facts 
will clarify the relevance of this definition:
\begin{itemize}
  \item The LM condition is preserved by the free Dirac flow:
  \begin{equation*}
    \quad\text{if}\  
    f\in E
    \ \text{a.e., then}\  
    e^{it \mathcal{D}}f\in E
    \ \text{a.e. for all}\  t.
  \end{equation*}
  \item A function $f$ satisfies LM iff its
  \emph{chiral invariant} $\rho(f)$ vanishes.
  The chiral invariant is the quantity
  \begin{equation*}
    \textstyle
    \rho(f):=
    |\langle \beta f,f\rangle|^{2}+
    |\langle \alpha_{5}f,f\rangle|^{2},
    \qquad
    \alpha_{5}=
    \left(
    \begin{smallmatrix}
       0& 0& -i& 0\\
       0& 0& 0& -i\\
       i& 0& 0& 0\\
       0& i& 0& 0
    \end{smallmatrix}
    \right).
  \end{equation*}
  \item As a consequence of the previous two facts,
  if the initial data $f$ satisfy LM, then the free flow
  $e^{it \mathcal{D}}f$ is also a solution of the cubic NLD
  \begin{equation*}
    iu_{t}+\mathcal{D}u=\langle \beta u,u\rangle \beta u
    \quad
    (\equiv0).
  \end{equation*}
\end{itemize}
Then a natural conjecture is that small perturbations of initial
data satisfying LM give rise to global large solution of the
cubic Dirac equation. This is indeed the case, as proved by
Bachelot \cite{Bac} for small perturbations in the $H^{6}$ norm.
If we introduce the projection 
$P:\mathbb{C}^{4}\to E$ given by
\begin{equation}\label{eq:defP}
  P
  \begin{pmatrix}
     z_{1}\\
     z_{2}\\
     z_{3}\\
     z_{4}
  \end{pmatrix}
  =
  \begin{pmatrix}
     z_{1}+z_{4}\\
     z_{2}-z_{3}\\
     \overline{z}_{3}-\overline{z}_{2}\\
     \overline{z}_{1}+\overline{z}_{4}
  \end{pmatrix}
\end{equation}
then Bachelot's condition on the initial data can be written
simply 
\begin{equation*}
  \|(I-P)f\|_{H^{6}}\ll1.
\end{equation*}

We shall proove that a similar situation occurs also in presence
of a potential $V$, provided $V$ has a suitable structure.
Denote by $\mathcal{V}$ the subspace of $4 \times 4$
complex matrices $M\in M_{4}(\mathbb{C})$ of the form
\begin{equation*}\label{eq:specialV}
  M=
  \begin{pmatrix}
    a & z & w & 0 \\
    \overline{z} & b & 0 & w \\
    \overline{w} & 0 & -b & z \\
    0 & \overline{w} & \overline{z} & -a
  \end{pmatrix}
  \quad \text{for some}\ 
  a,b\in \mathbb{R}
  \ \text{and}\ 
  z,w\in \mathbb{C}.
\end{equation*}
The space $\mathcal{V}$ can be characterized
in the following equivalent way:
\begin{equation}\label{eq:equivdefV}
  M\in \mathcal{V}
  \quad\iff\quad
  M=M^{*}
  \quad\text{and}\quad 
  \overline{M}\gamma=-\gamma M,
\end{equation}
where $\gamma$ is defined in \eqref{eq:defgamma}.
Note that the Dirac matrix $\beta$ belongs to $\mathcal{V}$,
thus if in the decomposition \eqref{eq:decompV} we assume
$A_{1}=A_{2}=A_{3}=0$ and $V_{0}(x)\in \mathcal{V}$,
we have $V(x)\in \mathcal{V}$ for all $x$.

We are in position to state our final result:

\begin{theorem}[Global existence, large data]\label{the:GWPlarge}
  Assume $V=A_{0}\beta+V_{0}$ satisfies the conditions of
  Theorem \ref{the:smooDV}--(iii)
  and in addition $V_{0}(x)\in \mathcal{V}$ for all $x$.

  Then there exists $\epsilon_{0}>0$ such that, for any
  data $u_{0}\in \Lambda^{-s}_{\omega}H^{1}$ with 
  $\|(I-P)\Lambda^{s}_{\omega}u_{0}\|_{H^{1}}\le \epsilon_{0}$,
  Problem \eqref{eq:Dirac} has a unique global solution
  $u\in C H^{1}\cap L^{2}L^{\infty}$ with
  $\Lambda_{\omega}^{s}u\in L^{\infty} H^{1}$;
  moreover $u$ scatters to a free solution, i.e.,
  there exists $u_{+}\in \Lambda^{-s}_{\omega}H^{1}$
  such that
  \begin{equation*}
    \lim_{t \rightarrow \infty} 
    \| \Lambda^s_{\omega} u(t) - \Lambda^s_{\omega} e^{it(\D + V)} u_+ \|_{H^1} =0.
  \end{equation*}
  A similar result holds for $t\to-\infty$.
\end{theorem}

Since $u_0$ is not small, the Theorem implies the existence of global solutions and scattering for a suitable class of large data.
Note that the result depends heavily on
the special structure of the nonlinearity.
Indeed, if we replace the nonlinear term
$(\beta u, u) \beta u$ with $|u|^3 I_4$, it is possible to
construct data such that $P u_0 = 0$ and the solution blows up
in a finite time, even in the case $V(x)=0$ (see \cite{DAnOka}).
Note also that the static potential $A_{0}$ can be large.

There are many results for the cubic Dirac equation when $V(x)$ 
is the constant matrix $m \beta$, $m\ge0$
(see \cite{Can, Pec, EscVeg, MNNO, BejHer1, BejHer2, BouCan} 
and references therein). 
In particular, Machihara et al. \cite{MNNO} proved small data 
scattering in $H^1 (\R ^3)$ with some additional regularity 
in the angular variables; our paper is in part an
extension of theirs, and of \cite{CacDAn}, to the case of
a large potential depending on $x$.
Note that in the massless case
$\dot H^1(\R ^3)$ is the critical space for scaling.
The final results on the constant coefficient case are due to
Bejenaru and Herr \cite{BejHer1} and Bournaveas and Candy 
\cite{BouCan}, who proved small data scattering in $H^{1}$.

Global existence for large data is a much more difficult
problem, in part since the conserved Dirac energy is not 
positive definite.
In the one dimensional case, Candy \cite{Can} proved 
the global well-posedness by using the conservation 
of the $L^{2}$ mass only.
In the higher dimensional case, one does not expect
local well posedness in time for $L^{2}$ data, since
the critical norm is stronger. 

As mentioned above,
Bachelot \cite{Bac} showed global existence of large 
amplitude solutions, by assuming smallness only for
the Chiral invariant related to the Lochak-Majorana condition; 
taking $V=0$ in Theorem \ref{the:GWPlarge} we reobtain his result
and actually improve on his $H^{6}$ condition on the initial data.
Indeed, the main tool in \cite{Bac} was the
commutating vector field method, which requires rather high 
regularity of the data to be applied.
We finally recall that in \cite{CacDAn} a result similar
to Theorem \ref{the:GWPsmall} was proved, but only for
a \emph{small} potentias $V$.

The outline of the paper is the following.
Sections \ref{sec:smoodirac} and \ref{sec:str_est} are devoted to
dispersive estimates for the linear flow. In
Section \ref{sec:GWPsmall} we prove global existence
for small data, Theorem \ref{the:GWPsmall}.
In Section \ref{sec:conserved} we check that the chiral invariant
is preserved by the perturbed flow if the potential has the
appropriate structure, and we apply this result to prove
global existence of large solutions,
Theorem \ref{the:GWPlarge}, in the concluding
Section \ref{sec:GWPlarge}.

\section{Smoothing estimates for the perturbed Dirac system}
\label{sec:smoodirac}

We prove here a smoothing estimate for the operator
\begin{equation}\label{eq:diracop}
  \mathcal{D}+V,
  \qquad
  V=A \cdot \alpha+\beta V_{0}+V_{0}
\end{equation}
where $A=(A_{1},A_{2},A_{3}):\mathbb{R}^{3}\to \mathbb{R}^{3}$,
$A_{0}:\mathbb{R}^{3}\to \mathbb{R}$ and
$V_{0}=V_{0}^{*}:\mathbb{R}^{3}\to M_{4}(\mathbb{C})$.
The relevant spaces are the Banach spaces 
$\dot X,\dot Y, \dot Y^{*}$ with norms
\begin{equation*}
  \textstyle
  \|v\|_{\dot{X}}^2:=\sup_{R>0}\frac1{R^2}
  \int_{|x|=R}|v|^2dS
  \simeq
  \||x|^{-1}v\|^{2}_{\ell^{\infty}L^{\infty}L^{2}},
\end{equation*}
\begin{equation*}
  \textstyle
  \|v\|_{\dot{Y}}^2:=\sup_{R>0}\frac1{R}
  \int_{|x|\leq R}|v|^2dx
  \simeq
  \||x|^{-1/2}v\|_{\ell^{\infty}L^{2}}^{2},
  \qquad
  \|v\|_{\dot Y^{*}}\simeq
  \||x|^{1/2}v\|_{\ell^{1}L^{2}}.
\end{equation*}
Note that $\dot Y^{*}$ is the predual of $\dot Y$
and an homogeneous version
of the Agmon--H\"{o}rmander space $B$
(see \cite{agmhor}).
In the following statement, $B=[B_{jk}]_{j,k=1}^{3}$ denotes the
magnetic field, defined by
\begin{equation*}
  B_{jk}=\partial_{j}A_{k}-\partial_{k}A_{j}.
\end{equation*}

\begin{theorem}[Smoothing estimates for Dirac]\label{the:smoooestD}
  Assume Condition (V) is satisfied with $\sigma$
  small enough.
  Then the perturbed flow $e^{it (\mathcal{D}+V)}$
  satisfies: for any $\rho\in \ell^{2}L^{\infty}$,
  \begin{equation}\label{eq:smooDiracV}
    \|\rho|x|^{-1/2} e^{it(\mathcal{D}+V)}f\|_{L^{2}_{t}L^{2}_{x}}
    \lesssim
    \|\rho\|_{\ell^{2}L^{\infty}} \|f\|_{L^{2}},
  \end{equation}
  \begin{equation}\label{eq:smooDiracVnh}
    \textstyle
    \|\rho|x|^{-1/2}\int_{0}^{t}e^{i(t-s)(\mathcal{D}+V)}F(s)ds
    \|_{L^{2}_{t}L^{2}_{x}}
    \lesssim
    \|\rho\|_{\ell^{2}L^{\infty}}^{2}
    \|\rho^{-1}|x|^{1/2}F\|_{L^{2}_{t}L^{2}_{x}}.
  \end{equation}
  If in addition $V$ satisfies
  \begin{equation}\label{eq:Vderass}
    \rho^{-2}|x||\partial V|\le K <\infty
  \end{equation}
  then we have also the estimate
  \begin{equation}\label{eq:smooDiracVD}
    \|\rho|x|^{-1/2}\partial e^{it(\mathcal{D}+V)}f\|
    _{L^{2}_{t}L^{2}_{x}}
    \lesssim
    \|\rho\|_{\ell^{2}L^{\infty}}
    (1+\|\rho\|_{\ell^{2}L^{\infty}}^{2}K)
    \|f\|_{H^{1}}.
  \end{equation}
\end{theorem}

Note that the additional condition \eqref{eq:Vderass}
is implied by Condition (V) for large $x$ and it
only restricts the singularity of $V$ near 0.

The proof of the Theorem is based on a resolvent estimate
for the squared operator $(\mathcal{D}+V)^{2}$.
This produces a system of stationary Schr\"{o}dinger equations with
diagonal principal part, as detailed in the following sections.
Two different methods are necessary
in order to handle the large frequency and the
short frequency regimes.

For the next result we need to assume that the magnetic
potential $A=(A_{1},A_{2},A_{3})$ is \emph{small}, while
the scalar potential $A_{0}$ may still be large.
By absorbing $A \cdot \alpha$ in the term $V_{0}$,
we see that
it is sufficient to consider a potential $V$ of the form
\begin{equation*}
  V(x)=A_{0}\beta+V_{0}.
\end{equation*}

\begin{corollary}[]\label{cor:smooang}
  Assume $V$ and $\mathcal{D}+V$ satisfy the conditions
  of the previous Theorem with $V$ of the special form
  \begin{equation*}
    V(x)=A_{0}\beta+V_{0}.
  \end{equation*}
  In addition, assume that for some $s\in(1,2]$
  and some $\rho\in \ell^{2}L^{\infty}$
  \begin{equation}\label{eq:assang}
    \rho^{-2}|x|\|\Lambda^{s}V_{0}(|x| \omega)\|
    _{L^{2}_{\omega}(\mathbb{S}^{2})}\le \epsilon,
    \qquad
    \rho^{-2}|x|
    \|\Lambda^{s}\partial V(|x| \omega)\|
    _{L^{2}_{\omega}(\mathbb{S}^{2})}
    \in L^{\infty}.
  \end{equation}
  \begin{equation}\label{eq:assangA0}
    \rho^{-2}|x|
    \|x\wedge \partial A_{0}(|x| \omega)\|
    _{L^{\infty}_{\omega}(\mathbb{S}^{2})}
    +
    \rho^{-2}\bra{x}
    \|\Delta_{\mathbb{S}} A_{0}(|x| \omega)\|
    _{L^{\infty}_{\omega}(\mathbb{S}^{2})}    
    \in L^{\infty}
  \end{equation}
  Then if $\epsilon$ is sufficiently small,
  the following estimates hold:
  \begin{equation}\label{eq:smooDiracVang}
    \|\rho|x|^{-1/2}
    \Lambda^{s}_{\omega}e^{it(\mathcal{D}+V)}f\|
    _{L^{2}_{t}L^{2}_{x}}
    \lesssim
    \|\Lambda^{s}_{\omega}f\|_{L^{2}},
  \end{equation}
  \begin{equation}\label{eq:smooDiracVangnh}
    \textstyle
    \|\rho|x|^{-1/2}
    \Lambda^{s}_{\omega}
    \int_{0}^{t}e^{i(t-s)(\mathcal{D}+V)}F(s)ds
    \|_{L^{2}_{t}L^{2}_{x}}
    \lesssim
    \|\rho^{-1}|x|^{1/2}\Lambda^{s}_{\omega}
    F\|_{L^{2}_{t}L^{2}_{x}},
  \end{equation}
  \begin{equation}\label{eq:smooDiracVangd}
    \|\rho|x|^{-1/2}
    \partial \Lambda^{s}_{\omega}e^{it(\mathcal{D}+V)}f\|
    _{L^{2}_{t}L^{2}_{x}}
    \lesssim
    \|\Lambda^{s}_{\omega}f\|_{H^{1}}.
  \end{equation}
\end{corollary}

Note that for a radial scalar potential $A_{0}=A_{0}(|x|)$
assumption \eqref{eq:assangA0} is trivially satisfied.

\subsection{Large frequencies} \label{sub:largefreq}

We consider a $4$--dimensional system of
stationary Schr\"{o}dinger equations on $\mathbb{R}^{3}$
\begin{equation}\label{eq:system}
  I_{4}\Delta_{A}v
  +
  W(x)v
  +
  \sum_{j=1}^{3}Z_{j}(x)\,\partial^{A}_{j}v
  +zv
  =f,
  \qquad
  z\in \mathbb{C}
\end{equation}
where
$v=(v_{1},v_{2},v_{3},v_{4})$,
$W(x),Z_{j}(x):\mathbb{R}^{3}\to M_{4}(\mathbb{C})$
are square $4 \times 4$ matrices,
$I_{4}$ is the $4$--dimensional identity matrix,
$\Delta_{A}$ the magnetic laplacian
on $\mathbb{R}^{3}$
\begin{equation*}
  \textstyle
  \Delta_{A}=\sum_{j=1}^{3}(\partial_{j}+iA_{j})^{2},
  \qquad
  \partial_{j}=\frac{\partial}{\partial x_{j}}
\end{equation*}
and $A(x)=(A_{1}(x),A_{2}(x),A_{3}(x))$ is a vector of
real valued functions. We also use the notations
\begin{equation*}
  \textstyle
  \partial_{j}^{A}=\partial_{j}+iA_{j}(x),
  \quad
  \partial=(\partial_{1},\partial_{2},\partial_{3}),
  \quad
  \partial^{A}=(\partial_{1}^{A},\partial_{2}^{A},\partial_{3}^{A})
\end{equation*}
and, writing
$\widehat{x}_{j}=\frac{x_{j}}{|x|}$
and
$\widehat{x}=\frac{x}{|x|}$,
\begin{equation*}
  B_{jk}=\partial_{j}A_{k}-\partial_{k}A_{j},
  \qquad
  \widehat{B}_{j}=B_{jk}\widehat{x}_{k},
  \qquad
  \widehat{B}=(\widehat{B}_{1},\widehat{B}_{2},\widehat{B}_{3}).
\end{equation*}
Here and in the following we use the convention
of implicit summation over repeated indices.

We begin by studying the case of large frequency
$|\Re z|\gg1$. In this regime we use a direct
approach, via the Morawetz multiplier method.

\begin{proposition}[Resolvent estimate for large frequencies]  \label{pro:resestlarge}
  There exists a constant $\sigma_{0}$ such that 
  the following holds.

  Let 
  $W_{L}(x),W_{S}(x),Z_{j}(x):
    \mathbb{R}^{3}\to M_{4}(\mathbb{C})$
  be square $4 \times 4$ matrices, let $W=W_{L}+W_{S}$,
  and let $v,f:\mathbb{R}^{3}\to \mathbb{C}^{4}$ satisfy
  \eqref{eq:system}.
  Assume that $|\Im z|\le1$ and
  \begin{equation}\label{eq:asscoeff}
    \||x|^{3/2}W_{S}\|_{\ell^{1}L^{2}L^{\infty}}
    +\||x|Z\|_{\ell^{1}L^{\infty}}
    \le \sigma_{0},
    \quad
    |\Re z|
    \ge
    \sigma_{0}^{-1}
    \left[
    \||x|\widehat{B}\|_{\ell^{1}L^{\infty}}^{2}
    +
    \||x|W_{L}\|_{\ell^{1}L^{\infty}}
    \right]
    +2.
  \end{equation}
  Then the following estimate holds
  \begin{equation}\label{eq:estlargela}
    \textstyle
    \|v\|_{\dot X}^{2}
    +
    |z|\|v\|_{\dot Y}^{2}
    +
    \|\partial^{A}v\|_{\dot Y}^{2}
    \lesssim
    \|f\|_{\dot Y^{*}}^{2}.
  \end{equation}
\end{proposition}

\begin{remark}[]\label{rem:deAtode}
  Under a weak additional assumption on $A$,
  the norm $\|\partial^{A}v\|_{\dot Y}$ in
  \eqref{eq:estlargela} can be replaced by 
  $\|\partial v\|_{\dot Y}$, thanks to the following
\end{remark}

\begin{lemma}[]\label{lem:estdeAde}
  Assume $A\in\ell^{\infty}L^{3}$.
  Then the following estimate holds
  \begin{equation}\label{eq:deAtode}
    \|\partial v\|_{\dot Y}
    \lesssim
    (1+\|A\|_{\ell^{\infty}L^{3}})
    \Bigl[\|\partial^{A}v\|_{\dot Y}+\|v\|_{\dot X}\Bigr]
  \end{equation}
  with an implicit constant independent of $A$.
\end{lemma}

\begin{proof}%[of ...]
  Let $C_{j}$ be the spherical shell $2^{j}\le|x|\le 2^{j+1}$
  and $\widetilde{C}_{j}=C_{j-1}\cup C_{j}\cup C_{j+1}$.
  Let $\phi$ be a nonnegative
  cutoff function equal to $1$ on $C_{j}$
  and vanishing outside $\widetilde{C}_{j}$, and let
  $\phi_{j}(x)=\phi(2^{-j}x)$. Then we can write
  \begin{equation*}
    \|\partial v\|_{L^{2}(C_{j})}\le
    \|\phi_{j}\partial v\|_{L^{2}}\le
    \|\phi_{j}\partial^{A} v\|_{L^{2}}
    +
    \|\phi_{j}Av\|_{L^{2}}
  \end{equation*}
  By H\"{o}lder's inequality and Sobolev embedding we have
  \begin{equation*}
    \|\phi_{j}Av\|_{L^{2}}\le
    \|A\|_{L^{3}(\widetilde{C}_{j})}
    \|\phi_{j}v\|_{L^{6}}
    \lesssim
    \|A\|_{\ell^{\infty}L^{3}}
    \|\phi_{j}v\|_{L^{6}}
    \lesssim
    \|A\|_{\ell^{\infty}L^{3}}
    \|\partial(\phi_{j}|v|)\|_{L^{2}}.
  \end{equation*}
  We expand the last term as
  \begin{equation*}
    \|\partial(\phi_{j} |v|)\|_{L^{2}}
    \le
    \|(\partial\phi_{j})|v|\|_{L^{2}}
    +
    \|\phi_{j}(\partial|v|)\|_{L^{2}}.
  \end{equation*}
  We note that $|\partial \phi_{j}|\lesssim 2^{-j}$
  and we recall the pointwise diamagnetic inequality
  \begin{equation*}
    |\partial|v||\le|\partial^{A}v|
  \end{equation*}
  valid since $A\in L^{2}_{loc}$. Then we can write
  \begin{equation*}
    \|\partial(\phi_{j}v)\|_{L^{2}}
    \lesssim
    2^{-j}\|v\|_{L^{2}(\widetilde{C}_{j})}
    +
    \|\partial^{A}v\|_{L^{2}(\widetilde{C}_{j})}
    \lesssim
    2^{-j/2}\|v\|_{L^{\infty} L^{2}(\widetilde{C}_{j})}
    +
    \|\partial^{A}v\|_{L^{2}(\widetilde{C}_{j})}.
  \end{equation*}
  Summing up, we have proved
  \begin{equation*}
    \|\partial v\|_{L^{2}(C_{j})}
    \lesssim
    (1+\|A\|_{\ell^{\infty}L^{3}})
    \left[
    \|\partial^{A}v\|_{L^{2}(\widetilde{C}_{j})}
    +
    2^{-j/2}\|v\|_{L^{\infty} L^{2}(\widetilde{C}_{j})}.
    \right]
  \end{equation*}
  Multiplying both sides by $2^{-j/2}$ and taking the sup
  in $j\in \mathbb{Z}$ we get the claim.
\end{proof}

\subsection{Large frequencies: formal identities}
\label{sub:for_ide}
In the course of the proof we shall reserve the symbols
\begin{equation*}
  \lambda=\Re z,
  \qquad
  \epsilon=\Im z
\end{equation*}
for the components of the frequency $z=\lambda+i \epsilon$
in \eqref{eq:system}.

The main tools are a few Moraw\-etz type identities,
based on the two multipliers
\begin{equation*}
  \overline{[\Delta_{A},\psi]w}=
  (\Delta \psi)\overline{w}+2 \partial \psi \cdot 
  \overline{\partial^{A}w}
  \quad\text{and}\quad 
  \phi \overline{w}
\end{equation*}
where $\phi(x),\psi(x)$ are real valued, spherically
symmetric weight functions to be chosen in the
following, and $w\in H^{2}_{loc}(\mathbb{R}^{3})$
is complex valued.
Define, with $c(x)$ a complex valued function,
\begin{equation*}
  \textstyle
  Q_{j}:=
  \partial_{j}^{A}w \, \overline{[\Delta_{A},\psi]w}
  -\frac12 \partial_{j}\Delta \psi\, |w|^{2}
  -\partial_{j}\psi\, 
  [c(x)|w|^{2}+|\partial^{A}w|^{2}]
\end{equation*}
and
\begin{equation*}
  \textstyle
  P_{j}:=\partial^{A}_{j}w\, \overline{w}\,\phi
  -\frac12 \partial_{j}\phi|w|^{2}
\end{equation*}
Then the following identities hold
\begin{equation}\label{eq:id1}
\begin{split}
  \Re \partial_{j}Q_{j}
  =
  &
  \Re(\Delta_{A}w-cw)\overline{[\Delta_{A},\psi]w}
  \textstyle
  -\frac12 \Delta^{2}\psi|w|^{2}
  +2 \partial^{A}_{j}w\, (\partial_{j}\partial_{k}\psi) \,
  \overline{\partial^{A}_{k}w}
  \\
  &
  -\Re (\partial_{j} \psi \, \overline{\partial_{j} c})|w|^{2}
  +2\Im( \overline{w}\,B_{jk}\,\partial^{A}_{j}w \,\partial_{k}\psi)
  -2(\Im c)\Im(w\,\partial_{j} \psi \, \overline{\partial^{A}_{j}w})
\end{split}
\end{equation}
and
\begin{equation}\label{eq:id2}
  \textstyle
  \partial_{j}P_{j}=
  \overline{w}\, \Delta_{A}w\,\phi
  +|\partial^{A}w|^{2}\, \phi
  -\frac12 \Delta \phi\,|w|^{2}
  +i\Im(\partial^{A}_{j}w\,\partial_{j}\phi\,\overline{w})
\end{equation}
These Morawetz type identities are well
known (see e.g.~\cite{CacDanLuc} for the form used here),
and are not difficult to check directly by expanding
the derivatives of $Q_{j},P_{j}$ at the left hand side
and keeping track of the resulting terms.

We need to apply the previous identities to a
$4$--tuple of functions $v=(v_{\alpha})_{\alpha=1}^{4}$.
We shall use the notation 
$|v|^{2}=|v_{1}|^{2}+ \dots +|v_{4}|^{2}$
and follow the convention of implicit summation over
repeated index $\alpha=1,\dots,4$.
If we define
\begin{equation*}
  g_{\alpha}:=\Delta_{A}v_{\alpha}+
  (\lambda+i \epsilon)v_{\alpha},
  \qquad
  \alpha=1,\dots,4
\end{equation*}
and denote by $Q_{j}^{\alpha}$, $P_{j}^{\alpha}$ the
quantities $Q_{j},P_{j}$ with $w$ replaced by $v_{\alpha}$,
we obtain
\begin{equation}\label{eq:fundid}
  \textstyle
  \Re \partial_{j}
  \{\sum_{\alpha}(Q_{j}^{\alpha}+P_{j}^{\alpha})\}
  =
  I_{\nabla v}+I_{v}+I_{\epsilon}+I_{B}+I_{g}
\end{equation}
where
\begin{equation*}
  I_{\nabla v}=
  2 \partial^{A}_{j}v_{\alpha}
  \, (\partial_{j}\partial_{k}\psi) \,
    \overline{\partial^{A}_{k}v_{\alpha}}
  + \phi|\partial^{A}v|^{2},
  \qquad
  \textstyle
  I_{v}=
  -\frac12 \Delta(\Delta \psi+\phi)|v|^{2}
  -\lambda \phi|v|^{2}
\end{equation*}
\begin{equation*}
  I_{B}=
  2\Im( \overline{v_{\alpha}}\,B_{jk}
  \,\partial^{A}_{j}v_{\alpha} 
  \,\partial_{k}\psi),
  \qquad
  I_{\epsilon}=
  2 \epsilon\Im(v_{\alpha}\,\partial_{j} 
  \psi \, \overline{\partial^{A}_{j}v_{\alpha}}),
\end{equation*}
\begin{equation*}
  I_{g}=
  \Re(g_{\alpha}\,\overline{[\Delta_{A},\psi]v_{\alpha}}
  +g_{\alpha}\,\overline{v_{\alpha}}\,\phi)
\end{equation*}

\subsection{Large frequencies:
preliminary estimates}\label{sub:pre_est}

We begin with a few elementary estimates based on
identity \eqref{eq:id2}, with different choices of the
radial weight $\phi$.
Writing \eqref{eq:id2} with $\phi=1$ and taking the
imaginary part, we get
\begin{equation*}
  \epsilon|v|^{2}=
  \Im(g_{\alpha}\overline{v}_{\alpha})
  -\Im \partial_{j}\{\overline{v}_{\alpha}
  \,\partial_{j}^{A}v_{\alpha}\}
\end{equation*}
and after integration on $\mathbb{R}^{3}$ we obtain
\begin{equation}\label{eq:aux1}
  \textstyle
  \epsilon\|v\|_{L^{2}}^{2}=
  \Im\int g_{\alpha}\overline{v}_{\alpha}.
\end{equation}
(Here and in the following we shall freely use the fact
that the boundary term vanish after integration, as it is
easy to check.)
Taking instead the real part of the same identity 
(with $\phi=1$) we obtain
\begin{equation*}
  |\partial^{A}v|^{2}=
  \lambda|v|^{2}-\Re(g_{\alpha}\overline{v}_{\alpha})
  +\Re\partial_{j}\{\overline{v}_{\alpha}
  \,\partial_{j}^{A}v_{\alpha}\}
\end{equation*}
and after integration
\begin{equation}\label{eq:aux2}
  \textstyle
  \|\partial^{A}v\|_{L^{2}}^{2}
  =
  \lambda\|v\|_{L^{2}}^{2}
  -\Re\int g_{\alpha}\overline{v}_{\alpha}.
\end{equation}

In order to estimate the term $I_{\epsilon}$ we use
\eqref{eq:aux1} and \eqref{eq:aux2} as follows:
\begin{equation*}
  \textstyle
  \int I_{\epsilon}
  \le
  2 |\epsilon|\|\partial \psi\|_{L^{\infty}}
  \|v\|_{L^{2}}\|\partial^{A}v\|_{L^{2}}
  \le
  C|\epsilon|^{1/2}(\int|g_{\alpha}\overline{v}_{\alpha}|)^{1/2}
  (|\lambda|\|v\|^{2}_{L^{2}}+\int|g_{\alpha}\overline{v}_{\alpha}|)
  ^{1/2}
\end{equation*}
with $C=2\|\partial \psi\|_{L^{\infty}}$,
then again by \eqref{eq:aux1}
\begin{equation*}
  \textstyle
  \le
  C(\int|g_{\alpha}\overline{v}_{\alpha}|)^{1/2}
  (|\lambda|\int|g_{\alpha}\overline{v}_{\alpha}|
  +|\epsilon|\int|g_{\alpha}\overline{v}_{\alpha}|)^{1/2}
\end{equation*}
and we arrive at the estimate
\begin{equation}\label{eq:estIep}
  \textstyle
  \int I_{\epsilon}
  \le
  2\|\partial \psi\|_{L^{\infty}}(|\lambda|+|\epsilon|)^{1/2}
  \|g_{\alpha}\overline{v}_{\alpha}\|_{L^{1}}.
\end{equation}

Another auxiliary estimate will cover 
the (easy) case of negative
$\lambda=-\lambda_{-}\le0$. Write 
the real part of identity \eqref{eq:id2} in the form
\begin{equation*}
  \textstyle
  \lambda_{-}|v|^{2}\phi+|\partial^{A}v|^{2}\phi
  -\frac12 \Delta \phi|v|^{2}
  =
  \sum_{\alpha}
  \partial_{j}\Re P_{j}^{\alpha}
  -\Re(g_{\alpha}\overline{v}_{\alpha})\phi
\end{equation*}
and choose the radial weight
\begin{equation*}
  \textstyle
  \phi=\frac{1}{|x|\vee R}
  \quad\implies\quad
  \phi'=-\frac{1}{|x|^{2}}\one{|x|>R},
  \quad
  \phi''=-\frac{1}{R^{2}}\delta_{|x|=R}+
    \frac{2}{|x|^{3}}\one{|x|>R}.
\end{equation*}
Note that
\begin{equation*}
  \textstyle
  -\Delta \phi=
  \frac{1}{R^{2}}\delta_{|x|=R}.
\end{equation*}
Integrating over $\mathbb{R}^{3}$ and taking the
supremum over $R>0$ we obtain the estimate
\begin{equation}\label{eq:aux4}
  \textstyle
  \lambda_{-}\|v\|_{\dot Y}^{2}
  +
  \|\partial^{A}v\|_{\dot Y}^{2}
  +
  \frac12\|v\|_{\dot X}^{2}
  \le
  \||x|^{-1}g_{\alpha}\overline{v}_{\alpha}\|_{L^{1}}.
\end{equation}

\subsection{Large frequencies: the main terms}\label{sub:mai_est}

In the following we assume 
$|\epsilon|\le1$ and $\lambda\ge2$.
We choose in \eqref{eq:fundid}, for arbitrary $R>0$,
\begin{equation}\label{eq:ourpsi}
  \psi=
  \frac{1}{2R}|x|^{2}\one{|x|\le R}+|x|\one{|x|>R},
  \qquad
  \phi=-\frac{1}{R}\one{|x|\le R}.
\end{equation}
We have then
\begin{equation}\label{eq:Apsifi}
  \psi'=\frac{|x|}{|x|\vee R},
  \qquad
  \psi''=
  \frac1R\one{|x|\le R},
  \qquad
  \Delta \psi+\phi=\frac{2}{|x|\vee R},
\end{equation}
\begin{equation*}
  \textstyle
  \Delta(\Delta \psi+\phi)=
    -\frac{2}{R^{2}}
    \delta_{|x|=R}
\end{equation*}
This implies
\begin{equation}\label{eq:estIv}
  \textstyle
  3
  \sup_{R>0}\int I_{v}
  \ge
  \|v\|_{\dot X}^{2}
  + \lambda\|v\|_{\dot Y}^{2}
\end{equation}

Next we can write, since $\psi$ is radial,
\begin{equation*}
  \textstyle
  2 \partial^{A}_{j}v_{\alpha}\, (\partial_{j}\partial_{k}\psi) \,
    \overline{\partial^{A}_{k}v_{\alpha}}
  =
  2\psi''
  \left|\widehat{x}\cdot \partial^{A}v_{\alpha}\right| ^{2}
  +
  2\frac{\psi'}{|x|}
  \left[|\partial^{A}v_{\alpha}|^{2}
    -\left|\widehat{x}\cdot \partial^{A}v_{\alpha}\right| ^{2}\right]
  \ge
  \frac{2}{R}\one{|x|<R}|\partial^{A}v_{\alpha}|^{2}.
\end{equation*}
This implies
\begin{equation}\label{eq:estInav}
  \textstyle
  \sup_{R>0}
  \int I_{\nabla v}
  \ge
  2\|\partial^{A}v\|_{\dot Y}^{2}.
  % + 2\||x|^{-1/2}(\partial^{A}v)_{T}\|_{L^{2}}^{2}
\end{equation}

Further we have, since 
$B_{jk}\partial_{k}\psi
  =B_{jk}\widehat{x}_{k}\psi'=\widehat{B}_{j}\psi'$,
\begin{equation*}
  |I_{B}|\le \frac{2|x|}{|x| \vee R}|v||\partial^{A}v||\widehat{B}|
  \le 2|v||\partial^{A}v||\widehat{B}|
\end{equation*}
which implies
\begin{equation*}
  \textstyle
  \int |I_{B}|\le
  2
  \||x|\widehat{B}\|_{\ell^{1}L^{\infty}}
  \||x|^{-1/2}\partial^{A}v\|_{\ell^{\infty}L^{2}}
  \||x|^{-1/2}v\|_{\ell^{\infty}L^{2}}
  =
  2
  \||x|\widehat{B}\|_{\ell^{1}L^{\infty}}
  \|\partial^{A}v\|_{\dot Y}
  \|v\|_{\dot Y}
\end{equation*}
and by Cauchy--Schwartz, for any $\delta>0$,
\begin{equation}\label{eq:estIB}
  \textstyle
  \int |I_{B}|\le
  \delta\|\partial^{A}v\|_{\dot Y}^{2}
  +
  \delta^{-1}
  \||x|\widehat{B}\|_{\ell^{1}L^{\infty}}^{2}
  \|v\|_{\dot Y}^{2}.
\end{equation}

Finally, since $|\Delta \psi+\phi|\le2|x|^{-1}$ and
$|\partial \psi|\le1$, we have
\begin{equation}\label{eq:estIg}
  \textstyle
  \int|I_{g}|\le
  2\||x|^{-1}g_{\alpha}\overline{v}_{\alpha}\|_{L^{1}}
  +
  2\|g_{\alpha}\overline{\partial^{A}v_{\alpha}}\|_{L^{1}}.
\end{equation}

Summing up, by integrating identity \eqref{eq:fundid} over
$\mathbb{R}^{3}$ and using estimates
\eqref{eq:estIep}
\eqref{eq:estIv}, \eqref{eq:estInav}, \eqref{eq:estIB}
and \eqref{eq:estIg} we obtain
(recall that $|\partial \psi|\le1$; recall also
that $\lambda\ge2$ and $|\epsilon|\le1$ so that
$|\epsilon|+|\lambda|\lesssim \lambda$)
\begin{equation*}
\begin{split}
  \textstyle
  \|v\|_{\dot X}^{2}
  + &
  \lambda\|v\|_{\dot Y}^{2}
  +
  \|\partial^{A}v\|_{\dot Y}^{2}
  \\
  \lesssim
  \delta\|\partial^{A}v\|_{\dot Y}^{2}
  +
  \delta^{-1}
  \||x|\widehat{B}\|_{\ell^{1}L^{\infty}}^{2}
  &
  \|v\|_{\dot Y}^{2}
  +
  \lambda^{1/2}
    \|g_{\alpha}\overline{v}_{\alpha}\|_{L^{1}}
  +
  \||x|^{-1}g_{\alpha}\overline{v}_{\alpha}\|_{L^{1}}
  +
  \|g_{\alpha}\overline{\partial^{A}v_{\alpha}}\|_{L^{1}}
\end{split}
\end{equation*}
where $\delta>0$ is arbitrary
and the implicit constant is a universal constant depending
only on $n,N$.
Note now that if $\delta$ is chosen small enough with 
respect to $n$ and we assume
\begin{equation}\label{eq:firstla}
  \lambda\ge c
  \||x|\widehat{B}\|_{\ell^{1}L^{\infty}}^{2}
\end{equation}
for a suitably large $c$, we can absorb two terms at the
right and we get the estimate
\begin{equation}\label{eq:firststep}
  \textstyle
  \|v\|_{\dot X}^{2}
  +
  \lambda\|v\|_{\dot Y}^{2}
  +
  \|\partial^{A}v\|_{\dot Y}^{2}
  \le
  c_{0}\left(
  \|\frac{g_{\alpha}\overline{v}_{\alpha}}{|x|}\|_{L^{1}}
  +
  \|g_{\alpha}\overline{\partial^{A}v_{\alpha}}\|_{L^{1}}
  +  \lambda^{\frac12}
    \|g_{\alpha}\overline{v}_{\alpha}\|_{L^{1}}
    \right)
\end{equation}
where $c_{0}\ge1$ is a universal constant.

\subsection{Large frequencies:
conclusion} \label{sub:the_arg_for_lar_fre}

We now define, for $v=(v_{\alpha})_{\alpha=1}^{4}$,
\begin{equation*}
  f:=I_{4}(\Delta_{A}+(\lambda+i \epsilon))v
  +
  W(x)v
  +
  Z(x)\cdot\partial^{A}v
\end{equation*}
where $Z=(Z_{1},Z_{2},Z_{3})$ and $W(x),Z_{j}(x)$ 
are $4 \times 4$ matrices.
We can apply estimate \eqref{eq:firststep} by defining
$g=(g_{1},\dots,g_{4})$ as
\begin{equation*}
  g=f-W(x)v- Z(x)\cdot\partial^{A}v.
\end{equation*}
We now estimate the terms at the right in 
\eqref{eq:firststep}, assuming that $W$ has a
(small) short range component and a (large) long range
component:
\begin{equation*}
  W=W_{S}+W_{L}.
\end{equation*}
We denote by $\gamma,\Gamma$ the quantities
\begin{equation*}
  \gamma:=\||x|^{3/2}W_{S}\|_{\ell^{1}L^{2}L^{\infty}}
  +\||x|Z\|_{\ell^{1}L^{\infty}},
  \qquad
  \Gamma:=\||x|W_{L}\|_{\ell^{1}L^{\infty}}.
\end{equation*}
Then we have
(we omit for simplicity the index $\alpha$)
\begin{equation*}
  \||x|^{-1}gv\|_{L^{1}}
  \le
  \||x|^{-1}W(x)v^{2}\|_{L^{1}}
  +
  \||x|^{-1} v Z(x)\cdot \partial^{A}v\|_{L^{1}}
  +
  \||x|^{-1}f\overline{v}\|_{L^{1}}
\end{equation*}
and, for any $\delta>0$,
\begin{equation*}
\begin{split}
  \||x|^{-1}W(x)v^{2}\|_{L^{1}}
  \le&
  \||x|^{1/2}W_{L}\|_{\ell^{1}L^{2}L^{\infty}}
  \|v\|_{\dot X}
  \|v\|_{\dot Y}
  +
  \||x|W_{S}\|_{\ell^{1}L^{1}L^{\infty}}\|v\|_{\dot X}^{2}
  \\
  \le&
  (\delta+\gamma)\|v\|_{\dot X}^{2}+
  \delta^{-1}\Gamma^{2}\|v\|_{\dot Y}^{2}
\end{split}
\end{equation*}
\begin{equation*}
  \||x|^{-1} v Z(x)\cdot \partial^{A}v\|_{L^{1}}
  \le
%  \||x|Z\|_{\ell^{1}L^{\infty}}
%  \|v\|_{\dot Y}\|\partial^{A}v\|_{\dot Y}
%  \le
%  \gamma^{2}\|\partial^{A}v\|_{\dot Y}^{2}
%  +
%  \gamma^{2}\|v\|_{\dot Y}^{2}
  \||x|^{1/2}Z\|_{\ell^{1}L^2L^{\infty}}
  \|v\|_{\dot X}\|\partial^{A}v\|_{\dot Y}
  \le
  \gamma \| v \|_{\dot X}^{2}
  +
  \gamma \|\partial^{A}v\|_{\dot Y}^{2}
\end{equation*}
\begin{equation*}
  \||x|^{-1}f\overline{v}\|_{L^{1}}
  \le
  \|f\|_{\dot Y^{*}}\|v\|_{\dot X}
  \le \delta\|v\|_{\dot X}^{2}
  +\delta^{-1}\|f\|_{\dot Y^{*}}^{2}.
\end{equation*}
In a similar way we have
\begin{equation*}
  \|g\partial^{A} v\|_{L^{1}}
  \le
  \|W(x)v\partial^{A} v\|_{L^{1}}
  +
  \|Z(x) (\partial^{A}v)^{2}\|_{L^{1}}
  +
  \|f\overline{\partial^{A}v}\|_{L^{1}},
\end{equation*}
and
\begin{equation*}
\begin{split}
  \|W(x)v\partial^{A} v\|_{L^{1}}
  \le&
  \||x|W_{L}\|_{\ell^{1}L^{\infty}}
  \|v\|_{\dot Y}\|\partial^{A}v\|_{\dot Y}
  +
  \||x|^{3/2}W_{S}\|_{\ell^{1}L^{2}L^{\infty}}
  \|v\|_{\dot X}\|\partial^{A}v\|_{\dot Y}
  \\
  \le&
  2\delta\|\partial^{A}v\|_{\dot Y}^{2}
  +
  \delta^{-1}\Gamma^{2}\|v\|_{\dot Y}^{2}
  +
  \delta^{-1}\gamma^{2}\|v\|_{\dot X}^{2},
\end{split}
\end{equation*}
\begin{equation*}
  \|Z(x) (\partial^{A}v)^{2}\|_{L^{1}}
  \le
  \||x|Z\|_{\ell^{1}L^{\infty}}
  \|\partial^{A}v\|_{\dot Y}^{2}
  \le
  \gamma\|\partial^{A}v\|_{\dot Y}^{2},
\end{equation*}
\begin{equation*}
  \|f\overline{\partial^{A}v}\|_{L^{1}}
  \le
  \delta\|\partial^{A}v\|_{\dot Y}^{2}
  +
  \delta^{-1}\|f\|_{\dot Y^{*}}^{2}.
\end{equation*}
Finally we have
\begin{equation*}
  \lambda^{1/2}\|g \overline{v}\|_{L^{1}}
  \le
  \lambda^{1/2}\|W v^{2}\|_{L^{1}}
  +
  \lambda^{1/2}\|Z v\partial^{A}v\|_{L^{1}}
  +
  \lambda^{1/2}\|f \overline{v}\|_{L^{1}}
\end{equation*}
and
\begin{equation*}
\begin{split}
  \lambda^{1/2}\|W v^{2}\|_{L^{1}}
  \le&
  \lambda^{1/2}
  \||x|W_{L}\|_{\ell^{1}L^{\infty}}\|v\|_{\dot Y}^{2}
  +
  \lambda^{1/2}
  \||x|^{3/2}W_{S}\|_{\ell^{1}L^{2}L^{\infty}}
  \|v\|_{\dot X}\|v\|_{\dot Y}
  \\
  \le&
  (\lambda^{1/2}\Gamma+\lambda \gamma)\|v\|_{\dot Y}^{2}
  +
  \gamma\|v\|_{\dot X}^{2},
\end{split}
\end{equation*}
\begin{equation*}
  \lambda^{1/2}\|Z v\partial^{A}v\|_{L^{1}}
  \le
  \lambda^{1/2}
  \||x|Z\|_{\ell^{1}L^{\infty}}
  \|v\|_{\dot Y}\|\partial^{A}v\|_{\dot Y}
  \le
  \lambda \gamma\|v\|_{\dot Y}^{2}
  +
  \gamma \|\partial^{A}v\|_{\dot Y}^{2},
\end{equation*}
\begin{equation*}
  \lambda^{1/2}\|f \overline{v}\|_{L^{1}}
  \le
  \delta \lambda\|v\|_{\dot Y}^{2}
  +
  \delta^{-1}\|f\|_{\dot Y^{*}}^{2}.
\end{equation*}
Summing up, we get
\begin{equation*}
\begin{split}
  \||x|^{-1}gv\|_{L^{1}}
  +&
  \|g\partial^{A} v\|_{L^{1}}
  +
  \lambda^{1/2}\|g \overline{v}\|_{L^{1}}
  \le
  (2\delta+3\gamma+\delta^{-1}\gamma^{2})\|v\|_{\dot X}^{2}
  \\
  +
  (2\delta^{-1}
  &
  \Gamma^{2}+
    \lambda^{1/2}\Gamma+2\lambda \gamma+\delta \lambda)
    \|v\|_{\dot Y}^{2}
  +
  3 (\delta + \gamma) \|\partial^{A}v\|_{\dot Y}^{2}
  +
  3\delta^{-1}\|f\|_{\dot Y^{*}}^{2}
\end{split}
\end{equation*}
Recalling that $c_{0}\ge1$ is the constant in
\eqref{eq:firststep}, we require that
\begin{equation}\label{eq:constrconst}
  \textstyle
  \delta=\frac{1}{16c_{0}},
  \qquad
  \gamma\le\frac{1}{16c_{0}},
  \qquad
  |\lambda|\ge2^{8}c_{0}^{2}\Gamma^{2}+2
  +c\||x|\widehat{B}\|_{\ell^{1}L^{\infty}}^{2}
\end{equation}
(note that this implies also \eqref{eq:firstla}
and $\lambda\ge2$)
and one checks that
\begin{equation*}
  \textstyle
  2\delta+3\gamma+\delta^{-1}\gamma^{2}\le \frac{1}{2c_{0}},
  \qquad
  3(\delta+\gamma) \le \frac{1}{2c_{0}}
\end{equation*}
and
\begin{equation*}
  \textstyle
  2\delta^{-1} \Gamma^{2}+
    \lambda^{1/2}\Gamma+2\lambda \gamma+\delta \lambda
  \le \frac{\lambda}{2c_{0}}.
\end{equation*}
Thus with the choices \eqref{eq:constrconst} we have
\begin{equation*}
  \textstyle
  \||x|^{-1}gv\|_{L^{1}}
  +
  \|g\partial^{A} v\|_{L^{1}}
  +
  \lambda^{1/2}\|g \overline{v}\|_{L^{1}}
  \le
  \frac{1}{2c_{0}}\|v\|_{\dot X}^{2}
  +
  \frac{\lambda}{2c_{0}} \|v\|_{\dot Y}^{2}
  +
  \frac{1}{2c_{0}}\|\partial^{A}v\|_{\dot Y}^{2}
  +
  3\delta^{-1}\|f\|_{\dot Y^{*}}^{2}
\end{equation*}
and plugging this into \eqref{eq:firststep}, and absorbing
the first three terms at the right from the left side of the 
inequality, we conclude that
\begin{equation}\label{eq:secondstep}
  \textstyle
  \|v\|_{\dot X}^{2}
  +
  \lambda\|v\|_{\dot Y}^{2}
  +
  \|\partial^{A}v\|_{\dot Y}^{2}
  \le
  c_{1}\|f\|_{\dot Y^{*}}^{2}.
\end{equation}

Note that if we consider the case of \emph{negative}
$\lambda$, starting from 
estimate \eqref{eq:aux4} instead of \eqref{eq:firststep}
and applying the same argument, we obtain a similar 
estimate, provided $\lambda$ satisfies
\eqref{eq:constrconst}. Recalling also that by assumption
$|\epsilon|\le|\lambda|$, we see that the proof of
Proposition \ref{pro:resestlarge} is concluded.

\subsection{Small frequencies} \label{sub:the_sma_fre_reg}

We now consider the remaining case of small requencies.
In this region we shall follow an indirect approach.
We consider an operator $L$ defined by
\begin{equation}\label{eq:defL}
  Lv:=
  -I_{4}\Delta v
  -
  W(x)v
  -
  i\sum_{j=1}^{3}Z_{j}^{*}(x)\,\partial_{j}v
  -
  i\sum_{j=1}^{3}\partial_{j}(Z_{j}(x)v)
\end{equation}
with $W=W^{*}$, and we assume that $L$
is selfadjoint on $L^{2}(\mathbb{R}^{3};\mathbb{C}^{4})$.
(Note that in the case of small frequencies it is not
useful to handle the magnetic part $A$ of the potential
separately).
In order to estimate the resolvent operator of $L$
\begin{equation*}
  R(z):=(L-z)^{-1}=
  (-I_{4}\Delta-W-iZ^{*} \cdot \partial-i \partial \cdot Z-z)^{-1}
\end{equation*}
we use the (Lippmann--Scwinger) representation of $R(z)$
\begin{equation}\label{eq:LSE}
  R(z)=R_{0}(z)(I_{4}-K(z))^{-1},
  \qquad
  K(z):=[W+iZ^{*} \cdot \partial+i \partial \cdot Z]R_{0}(z)
\end{equation}
in terms of the free resolvent
\begin{equation*}
  R_{0}(z)=I_{4}(-\Delta-z)^{-1}.
\end{equation*}

We recall a few (more or less standard) facts on the
free resolvent $R_{0}(z)$. For 
$z\in \mathbb{C} \setminus[0,+\infty)$, $R_{0}(z)$
is a holomorphic map with values in the space of bounded 
operators $L^{2}\to H^{2}$ and satisfies an estimate
\begin{equation}\label{eq:freeest}
  \|R_{0}(z)f\|_{\dot X}
  +
  |z|^{\frac12}\|R_{0}(z)f\|_{\dot Y}
  +
  \|\partial R_{0}(z)f\|_{\dot Y}
  \lesssim
  \|f\|_{\dot Y^{*}}
\end{equation}
with an implicit constant independent of $z$
(a proof of this estimate is actually contained in the
previous section since for vanishing potentials there is no
restriction on $\lambda$; for a detailed proof see e.g.
\cite{CacDanLuc}). When $z$ approaches the spectrum
of the Laplacian $\sigma(-\Delta)=[0,+\infty)$, it is possible
to define two limit operators
\begin{equation*}
  R(\lambda\pm i 0)=
  \lim_{\epsilon \downarrow0}R(\lambda\pm i \epsilon),
  \qquad
  \epsilon>0, \lambda\ge0
\end{equation*}
but the two limits are different if $\lambda>0$. These limits
exist in the norm of bounded operators
from the weighted $L^{2}_{s}$ space with norm
$\|\bra{x}^{s}f\|_{L^{2}}$ to the weighted Sobolev space
$H^{2}_{-s'}$ with norm 
$\sum_{|\alpha|\le2}\|\bra{x}^{-s'}\partial^{\alpha}f\|_{L^{2}}$,
for arbitrary $s,s'>1/2$ (see \cite{Agmon75-a}).
Since these spaces are dense in $\dot Y^{*}$ and
$\dot Y$ (or $\dot X$) respectively, and estimate
\eqref{eq:freeest} is uniform in $z$, one obtains that
\eqref{eq:freeest} is valid also for the limit operators
$R_{0}(\lambda\pm i0)$. In the following we shall write
simply $R_{0}(z)$, $z\in \overline{\mathbb{C}^{\pm}}$, 
to denote either one of 
the extended operators $R_{0}(\lambda\pm i \epsilon)$ with
$\epsilon\ge0$, defined on the closed upper (resp.~lower) complex
half--plane. Note also that the map $z \mapsto R_{0}(z)$
is continuous with respect to the operator norm of
bounded operators $L^{2}_{s}\to H^{2}_{-s'}$, 
for every $s,s'>1/2$, and from this fact one easily
obtains that it is also continuous with respect to
the operator norm of bounded operators
from $\dot Y^{*}\to H^{2}_{-s'}$. 

Thus in particular
\begin{equation*}
  R_{0}(z):\dot Y^{*}\to \dot X,
  \qquad
  \partial R_{0}(z):\dot Y^{*}\to \dot Y
\end{equation*}
are uniformly bounded operators for all
$z\in \overline{\mathbb{C}^{\pm}}$;
note also the formula
\begin{equation*}
  \Delta R_{0}(z)=-I_{4}-z R_{0}(z)
\end{equation*}
Moreover, for any smooth cutoff
$\phi\in C^{\infty}_{c}(\mathbb{R}^{3})$ and all 
$z\in \overline{\mathbb{C}^{\pm}}$,
the map 
$z\mapsto\phi R_{0}(z)$ is 
continuous w.r.to the norm of bounded operators
$\dot Y^{*}\to H^{2}$,
and hence 
\begin{equation*}
  \phi R_{0}(z):\dot Y^{*}\to L^{2}
  \quad\text{and}\quad 
  \phi \partial R_{0}(z):\dot Y^{*}\to L^{2}
  \quad \text{are compact operators.}
\end{equation*}
Similarly one gets that
$z\mapsto\phi R_{0}(z)$ is 
continuous w.r.to the norm of bounded operators
$\dot Y^{*}\to L^{\infty}_{|x|} L^{2}_{\omega}$ and
\begin{equation*}
  \phi R_{0}(z):\dot Y^{*}\to L^{\infty}_{|x|} L^{2}_{\omega}
  \quad \text{is a compact operator.}
\end{equation*}

In order to invert the operator $I-K(z)$
we shall apply Fredholm theory. An essential step is the
following compactness result:

\begin{lemma}[]\label{lem:compact}
  Let $z\in \overline{\mathbb{C}^{\pm}}$ and assume
  $W,Z$ satisfy
  \begin{equation}\label{eq:assWZ}
    N:=
    \||x|^{3/2}(W+i( \partial \cdot Z))\|_{\ell^{1}L^{2}L^{\infty}}
    +
    \||x|Z\|_{\ell^{1}L^{\infty}}<\infty.
  \end{equation}
  Then
  $K(z)=(W+iZ^{*} \cdot \partial+i \partial \cdot Z)R_{0}(z)$
  is a compact operator on $\dot Y^{*}$,
  and the map $z \mapsto K(z)$ is continuous with respect
  to the norm of bounded operators on $\dot Y^{*}$.
  % $L^{2}_{\rho}$.
\end{lemma}

\begin{proof}%[of ...]
  We decompose $K$
  as follows. Let $\chi\in C^{\infty}_{c}(\mathbb{R}^{3})$
  be a cutoff function equal to 1 for $|x|\le1$ and to 0 for
  $|x|\ge2$. Define for $r>2$
  \begin{equation*}
    \chi_{r}(x)=\chi(x/r)(1-\chi(rx))
  \end{equation*}
  so that $\chi_{r}$ vanishes for $|x|\ge2r$ and also for
  $|x|\le1/r$, and equals 1 when $2/r\le|x|\le r$.
  Then we split
  \begin{equation*}
    K=A_{r}+B_{r}
  \end{equation*}
  where
  \begin{equation*}
    A_{r}(z)=\chi_{r}\cdot K(z),
    \qquad
    B_{r}(z)=(1-\chi_{r})\cdot K(z).
  \end{equation*}

  First we 
  show that $A_{r}$ is a compact operator on $\dot Y^{*}$.
  Indeed, for $s>2r>4$ we have $\chi_{r}\chi_{s}=\chi_{r}$ and
  we can write
  \begin{equation*}
    A_{r}=\chi_{s}A_{r}=
    \chi_{s}(W+i( \partial \cdot Z))\chi_{r}R_{0}(z)
    +
    i
    \chi_{s}(Z+Z^{*}) \cdot \chi_{r}\partial R_{0}(z).
  \end{equation*}
  By the estimate
  \begin{equation}\label{eq:estVN}
    \|(W+i( \partial \cdot Z))v\|_{\dot Y^{*}}\le
    \||x|^{3/2}(W+i( \partial \cdot Z))\|_{\ell^{1}L^{2}L^{\infty}}
    \|v\|_{\dot X}
    \le
    N\|v\|_{\dot X}
  \end{equation}
  we see that multiptlication by $W+i( \partial \cdot Z)$ is a bounded operator
  from $\dot X$ to $\dot Y^{*}$.
  Moreover, multiplication by $\chi_{s}$ is a bounded
  operator $L^{\infty}_{|x|}L^{2}_{\omega}\to \dot X$
  and the operator
  $\chi_{r}R_{0}:\dot Y^{*}\to L^{\infty}_{|x|}L^{2}_{\omega}$
  is compact as remarked above. 
  A similar argument applies to the second term in $A_{r}$,
  using the estimate
  \begin{equation}\label{eq:estZN}
    \|Zv\|_{\dot Y^{*}}
    \le
    \||x|Z\|_{\ell^{1}L^{\infty}}\|v\|_{\dot Y}
    \le N\|v\|_{\dot Y}
  \end{equation}
  and compactness of $\chi_{r}\partial R_{0}:\dot Y^{*}\to L^{2}$.
  Summing up, we obtain that $A_{r}:\dot Y^{*}\to \dot Y^{*}$
  is a compact operator.
  Similarly, we see that $z \mapsto A_{r}(z)$ is continuous
  with respect to the norm of bounded operators on $\dot Y^{*}$.

  Then to conclude the proof it is sufficient to show that
  $B_{r}\to0$ in the norm of bounded operators
  on $\dot Y^{*}$, uniformly in $z$.
  We have, as in
  \eqref{eq:estVN}--\eqref{eq:estZN},
  \begin{equation*}
    \|B_{r}v\|_{\dot Y^{*}}
    \le
    N_{r}
    (\|R_{0}\|_{\dot Y^{*}\to \dot X}+
    \|\partial R_{0}\|_{\dot Y^{*}\to \dot Y})
    \|v\|_{\dot Y^{*}}
  \end{equation*}
  where
  \begin{equation*}
    N_{r}:=
    \||x|^{3/2}(1-\chi_{r})(W+i( \partial \cdot Z))
      \|_{\ell^{1}L^{2} L^{\infty}}
    +
    2\||x|(1-\chi_{r})Z\|_{\ell^{1}L^{\infty}}.
  \end{equation*}
  Since $N_{r}\to0$ as $r\to \infty$,
  we obtain that $\|B_{r}\|_{\dot Y^{*}\to \dot Y^{*}}\to0$.
\end{proof}

We now study the injectivity of $I-K(z):\dot Y^{*}\to \dot Y^{*}$.
Note that if $f\in \dot Y^{*}$ satisfies
\begin{equation*}
  (I_{4}-K(z))f=0
\end{equation*}
then setting $v=R_{0}(z) f$ 
by the properties of $R_{0}(z)$ we have
$v\in H^{1}_{loc} \cap \dot X$, 
$\nabla v\in \dot Y$,
$v\in H^{2}_{loc}(\mathbb{R}^{3}\setminus0)$,
$\Delta v\in \dot Y+\dot Y^{*}$
(or $\Delta v\in\dot Y^{*}$ if $z=0$)
and if $z\neq0$ we have also $v\in\dot Y$.
In particular, $v$ is a solution of the
equation 
\begin{equation*}
  (L-z)v=0.
\end{equation*}
For $z$ outside the spectrum of $L$ it is easy to check
that this implies $v=f=0$:

\begin{lemma}[]\label{lem:eigvcomplex}
  Let $W,Z,K(z)$ be as in Lemma \ref{lem:compact} and
  $L=-I_{4}\Delta-W-iZ^{*} \cdot \partial-i \partial \cdot Z$.
  If $f\in \dot Y^{*}$ satisfies
  \begin{equation*}
    (I_{4}-K(z))f=0
  \end{equation*}
  for some $z\not\in \sigma(L)$, then $f=0$.
\end{lemma}

\begin{proof}%[of ...]
  Let $v=R_{0}(z)f$,
  fix a compactly supported smooth function $\chi$ which is
  equal to 1 for $|x|\le1$, and for $M>1$ consider
  $v_{M}:=v(x)\chi(x/M)$. Then $v_{M}\in L^{2}$ and
  \begin{equation*}
    \textstyle
    (L-z)v_{M}=
    \frac1M \nabla\chi(\frac xM)(2\nabla v+i(Z+Z^{*}) v)
    +\frac1{M^{2}} \Delta \chi(\frac xM)v
    =:f_{M}.
  \end{equation*}
  We have, for $\delta\in(1,\frac12)$,
  using the estimate $|Z|\lesssim|x|^{-1}$,
  \begin{equation*}
  \begin{split}
    \|f_{M}\|_{L^{2}}
    \lesssim
    &
    M^{\delta-\frac12}
    \left(
      \||x|^{-\frac12-\delta}\nabla v\|_{L^{2}(|x|\ge M)}
      +
      \||x|^{-\frac32-\delta}v\|_{L^{2}(|x|\ge M)}
    \right)
    \\
    \lesssim
    &
    M^{\delta-\frac12}
    (\|\nabla v\|_{\dot Y}+\|v\|_{\dot X})
    \end{split}
  \end{equation*}
  uniformly in $M$, so that $f_{M}\to 0$ in $L^{2}$
  as $M\to \infty$.
  Since $v_{M}=R_{0}(z)f_{M}$ and $R_{0}(z)$ is a bounded
  operator on $L^{2}$, we conclude that $v=f=0$.
\end{proof}

The hard case is of course $z\in \sigma(L)$. Then we have the
following result, in which we write simply
\begin{equation*}
  R_{0}(\lambda)
  \quad\text{instead of}\quad 
  R_{0}(\lambda\pm i0)
\end{equation*}
since the computations for the two cases are identical.

\begin{lemma}[]\label{lem:eigv}
  Assume $W=W^{*}$ and $Z$ satisfy for some $\delta>0$
  \begin{equation}\label{eq:WZdelta}
    \||x|^{2}\bra{x}^{\delta}(W+i\partial \cdot Z)\|
    _{\ell^{1}L^{\infty}}
    +
    \||x|\bra{x}^{\delta}Z\|_{\ell^{1}L^{\infty}}
    <\infty
  \end{equation}
  and 
  $L=-\Delta I_{4}-W-i \partial \cdot Z-iZ^{*} \cdot \partial $ 
  is a non negative selfadjoint operator on $L^{2}$.
  Let $f\in \dot Y^{*}$ be such that, for some
  $\lambda\ge0$,
  \begin{equation*}
    (I_{4}-K(\lambda))f=0,
    \qquad
    K(\lambda)
    :=(W+i \partial \cdot Z+iZ^{*} \cdot\partial)R_{0}(\lambda).
  \end{equation*}
  Then in the case $\lambda>0$ we have $f=0$,
  while in the case $\lambda=0$ we have
  $|x|^{3/2}f\in L^{2}$ and the function $v=R_{0}(0)f$
  belongs to $H^{2}_{loc}(\mathbb{R}^{3}\setminus0)\cap \dot X$
  with $\partial v\in \dot Y$,
  solves $Lv=0$ and satisfies 
  $|x|^{-\frac 12-\delta'}v\in L^{2}$  and
  $|x|^{\frac 12-\delta'}\partial v\in L^{2}$ 
  for any $\delta'>0$.
\end{lemma}

\begin{proof}%[of ...]
  Defining as in the previous proof $v=R_{0}(\lambda)f$,
  we see that $v$ solves
  \begin{equation}\label{eq:eqvg}
    \Delta I_{4}v+\lambda v+g=0,
    \qquad
    g:=Wv+i Z^{*} \cdot \partial v +i \partial \cdot Zv.
  \end{equation}
  Then given a radial function $\chi\ge0$ to be precised later, we
  apply again identities \eqref{eq:id1}, \eqref{eq:id2}
  with the choices
  \begin{equation*}
    \psi'=\chi,
    \qquad
    \phi=-\chi'
  \end{equation*}
  so that in particular $\Delta \psi+\phi=\frac{2}{|x|}\chi$.
  We sum the two identities and integrate on a ball
  $B(0,R)$; it is easy to check that the boundary terms
  tend to 0 as $R\to \infty$, provided $\chi$ does not grow
  to fast ($\chi(x)\lesssim|x|$ is enough).
  After straightforward computations
  (see Proposition 3.1 of \cite{CacDanLuc} for a similar
  argument), we arrive at the 
  following \emph{radiation estimate}:
  \begin{equation}\label{eq:radest}
    \textstyle
    \int \chi'|\partial_{S}v|^{2}
    +
    2(\frac{\chi}{|x|}-\chi')|(\partial v)_{T}|^{2}
    -
    \int \Delta(\frac{\chi}{|x|})|v|^{2}
    =
    \Re \int\chi g(\frac{2}{|x|}\overline{v}+
      2 \widehat{x}\cdot\overline{\partial_{S}v})
  \end{equation}
  where we denoted the "Sommerfeld" gradient of $v$ with
  \begin{equation*}
    \partial_{S}v:=\partial v-i \sqrt{\lambda}\widehat{x}v,
    \qquad
    \widehat{x}=x/|x|
  \end{equation*}
  and the tangential component of $\partial v$ with
  \begin{equation*}
    |(\partial v)_{T}|^{2}:=
    |\partial v|^{2}-|\widehat{x}\cdot \partial v|^{2}.
  \end{equation*}

  We now estimate the right hand side of \eqref{eq:radest}.
  We have
  \begin{equation*}
    \textstyle
    |\Re\int\chi g \frac{\overline{v}}{|x|}|
    \le
    \|\chi(W+i (\partial \cdot Z))|x|^{-1}|v|^{2}\|_{L^{1}}
    +
    2\|\chi Z|\partial v| |x|^{-1}v\|_{L^{1}}
  \end{equation*}
  \begin{equation*}
    \le
    \|\chi|x|(W+i(\partial \cdot Z))\|
      _{\ell^{1}L^{1}L^{2}}
    \|v\|_{\dot X}^{2}
    +
    2\|\chi|x|^{1/2}Z\|_{\ell^{1}L^{2}L^{\infty}}
    \|v\|_{\dot X}\|\nabla v\|_{\dot Y}
  \end{equation*}
  and similarly
  \begin{equation*}
    \textstyle
    |\Re\int \chi g \widehat{x}\cdot\overline{\partial_{S}v}|
    \le
    \|\chi(W+i(\partial \cdot Z)) v|\partial_{S}v|\|_{L^{1}}
    +
    \|\chi Z^{*} \cdot \partial v |\partial_{S}v|\|_{L^{1}}
  \end{equation*}
  \begin{equation*}
    \le
    \|\chi|x|^{3/2}(W+i(\partial \cdot Z))\|
      _{\ell^{1}L^{2}L^{\infty}}
    \|v\|_{\dot X}\|\partial_{S}v\|_{\dot Y}
    +
    \|\chi|x|Z\|_{\ell^{1}L^{\infty}}
    \|\partial v\|_{\dot Y}\|\partial_{S}v\|_{\dot Y}.
  \end{equation*}
  Since the quantities $\|v\|_{\dot X}$, $\|\partial v\|_{\dot Y}$
  and 
  $\|\partial_{S}v\|_{\dot Y}\le\|\partial v\|_{\dot Y}+
  \sqrt{\lambda}\|v\|_{\dot Y}$ are all estimated by
  $\|f\|_{\dot Y^{*}}$ (recall \eqref{eq:freeest}),
  we conclude
  \begin{equation}\label{eq:estgsmall}
    \textstyle
    \left|\Re \int\chi g(\frac{2}{|x|}\overline{v}+
      2 \widehat{x}\cdot\overline{\partial_{S}v})\right|
    \lesssim
    N_{\chi}^{2}\|f\|_{\dot Y^{*}}^{2}
  \end{equation}
  where
  \begin{equation*}
    N_{\chi}^{2}:=
    \|\chi|x|^{3/2}(W+i(\partial \cdot Z))\|
      _{\ell^{1}L^{2}L^{\infty}}
    +
    \|\chi|x|Z\|_{\ell^{1}L^{\infty}}.
  \end{equation*}
  Finally, if we choose 
  \begin{equation*}
    \chi(x)=|x|^{\delta}
    \quad\text{with}\quad 
    0<\delta\le 1
  \end{equation*}
  by \eqref{eq:radest} and \eqref{eq:estgsmall} we obtain,
  dropping a (nonnegative) term at the left,
  \begin{equation}\label{eq:finaldelta}
    \||x|^{(\delta-1)/2}\partial_{S}v\|_{L^{2}}
    +
    \||x|^{(\delta-3)/2}v\|_{L^{2}}
    \lesssim_{\delta}
    N_{\delta}\|f\|_{\dot Y^{*}}
  \end{equation}
  where by assumption
  \begin{equation*}
    N_{\delta}^{2}
    :=
    \||x|^{3/2+\delta}(W+i(\partial \cdot Z))\|
      _{\ell^{1}L^{2}L^{\infty}}
    +
    \||x|^{1+\delta}Z\|_{\ell^{1}L^{\infty}}<\infty.
  \end{equation*}

  Consider now the following identity, obtained using the
  divergence formula:
  \begin{equation*}
    \textstyle
    \int_{|x|= R}
    (|\partial v|^{2}+\lambda|v|^{2}-|\partial_{S}v|^{2})
    d \sigma
    =
    2\Re\int_{|x|\le R} i\sqrt{\lambda}\,
    \partial \cdot \langle v,\partial v\rangle
    =
    2\Re\int_{|x|\le R} i\sqrt{\lambda}\,
    \langle v,\Delta v\rangle
  \end{equation*}
  for arbitrary $R>0$. Substituting
  $\Delta v=-\lambda v-g$ from \eqref{eq:eqvg}
  and dropping two pure imaginary terms, we get
  \begin{equation*}
    \textstyle
    \int_{|x|= R}
    (|\partial v|^{2}+\lambda|v|^{2}-|\partial_{S}v|^{2})
    d \sigma
    =
    2\Re
    \int_{|x|\le R}
    \langle Z^{*} \cdot \partial v +\partial \cdot Zv,v\rangle
  \end{equation*}
  The last term can be written,
  again by the divergence formula,
  \begin{equation*}
    \textstyle
    =2\Re\int_{|x|\le R}\partial \langle Zv,v\rangle
    =2\sum_{j}\int_{|x|=R}
    \widehat{x}_{j}
    \langle Z_{j}v,v\rangle
    d \sigma,
    \qquad
    \widehat{x}_{j}=x_{j}/|x|.
  \end{equation*}
  By assumption $|Z|\lesssim|x|^{-1}$, hence for some $R_{0}>0$
  we have $\lambda>2|Z(x)|$ for all $|x|>R_{0}$, and the term in $Z$
  can be absorbed at the left of the identity. Summing up,
  we have proved that
  \begin{equation}\label{eq:equivna}
    \textstyle
    \int_{|x|=R}(|\partial v|^{2}+\lambda|v|^{2})d \sigma
    \le
    2\int_{|x|=R}|\partial_{S} v|^{2}d \sigma,
    \qquad
    R\ge R_{0}.
  \end{equation}
  Multiplying both sides by $|x|^{\delta-1}$,
  integrating in the radial direction from $R_{0}$ to $\infty$,
  and using \eqref{eq:finaldelta}, we conclude
  \begin{equation}\label{eq:deltaL2}
    \||x|^{(\delta-1)/2}\partial v\|_{L^{2}(|x|\ge R_{0})}
    +
    \sqrt{\lambda}
    \||x|^{(\delta-1)/2} v\|_{L^{2}(|x|\ge R_{0})}
    \lesssim
    \|f\|_{\dot Y^{*}}.
  \end{equation}

  In the case $\lambda>0$ we have proved that
  $|x|^{(\delta-1)/2}v\in L^{2}$ i.e., $\lambda$ is a resonance,
  and this is enough to conclude that $v=0$ by applying
  one of the available results on the absence of embedded
  eigenvalues. For instance, we can apply
  the results from \cite{KochTataru06-a} which are 
  partiularly sharp.
  Note that in \cite{KochTataru06-a}
  a \emph{scalar} operator is considered, but
  it is easy to check that the same proof covers also the case of 
  an operator which is diagonal in the principal part and
  coupled only in lower order terms.
  We need to check the assumptions on the potentials required
  in \cite{KochTataru06-a}.
  The potential $V$ in \cite{KochTataru06-a} is simply $V=z$ 
  in our case, which we are assuming real and $>0$, 
  thus condition A.1 is trivially satisfied.
  Concerning $W$ we have
  \begin{equation*}
    \|W\|_{L^{3/2}}
    \le
    \||x|^{-2}\|_{\ell^{\infty}L^{3/2}}
    \||x|^{2}W\|_{\ell^{3/2}L^{\infty}}<\infty
  \end{equation*}
  by assumption,
  thus $W\in L^{3/2}$ and condition A.2 in
  \cite{KochTataru06-a} is satisfied.
  Concerning the potential $Z$, we have
  \begin{equation*}
    \|Z\|_{\ell^{\infty}L^{3}}
    \le
    \||x|^{-1}\|_{\ell^{\infty}L^{3}}
    \||x|Z\|_{\ell^{3} L^{\infty}}<\infty
  \end{equation*}
  thus $Z\in \ell^{\infty}L^{3}$; moreover a similar computation
  applied to $\one{|x|>M}Z$ gives
  \begin{equation*}
    \|\one{|x|>M}Z\|_{\ell^{\infty}L^{3}}
    \le
    \||x|^{-1}\|_{\ell^{\infty}L^{3}}
    \|\one{|x|>M}|x|^{-\delta}\|_{L^{\infty}}
    \||x|^{1+\delta}Z\|_{\ell^{3} L^{\infty}}\to0
    \quad\text{as}\quad M\to \infty.
  \end{equation*}
  Thus to check that $Z$ satisfies condition A.3
  in \cite{KochTataru06-a} it remains to check that
  the low frequency part $S_{<R}Z$ of $Z$ satisfies A.2
  for $R$ large enough. $S_{<R}Z$ is obviously smooth.
  Moreover, it is clear that
  $|x|Z\to0$ as $|x|\to \infty$;
  in order to prove the same
  decay property for $S_{<R}Z$ we represent it as
  a convolution with a suitable Schwartz kernel $\phi$
  \begin{equation*}
    \textstyle
    \phi *Z(x)=
    \int_{|y|\le \frac{|x|}{2}}
    Z(y)\phi(x-y)
    +
    \int_{|y|\ge \frac{|x|}{2}}
    Z(y)\phi(x-y).
  \end{equation*}
  The first integral is bounded by $C_{k} \bra{x}^{-k}$ for all $k$.
  For the second one we write
  \begin{equation*}
    \textstyle
    |x|\int_{|y|\ge \frac{|x|}{2}}
    Z(y)\phi(x-y)
    \le
    \int_{|y|\ge \frac{|x|}{2}}
    |y|Z(y)\phi(x-y)=o(|x|).
  \end{equation*}
  We have thus proved that $|x|S_{<R}Z\to0$ as $|x|\to \infty$
  (for any fixed $R$) and hence $Z$ satisfies condition A.3.
  Applying Theorem 8 of \cite{KochTataru06-a}, we conclude
  that $v=0$.

  It remains to consider the case $\lambda=0$.
  We denote by $\dot L^{2,s}$ the Hilbert space with norm
  \begin{equation*}
    \|v\|_{\dot L^{2,s}}:=\||x|^{s}v\|_{L^{2}}.
  \end{equation*}
  By the well known Stein--Weiss estimate for fractional
  integrals in weighted $L^{p}$ spaces, applied to
  $R_{0}(0)v=\Delta^{-1}v=c|x|^{-1}*v$, we see that $R_{0}(0)$
  is a bounded operator
  \begin{equation*}
    \textstyle
    R_{0}(0):\dot L^{2,s}\to \dot L^{2,s-2}
    \qquad\text{for all}\qquad
    \frac 12<s<\frac 32
  \end{equation*}
  while $\partial R_{0}(0)=c (x|x|^{-3})*v$ is a bounded operator
  \begin{equation*}
    \textstyle
    \partial R_{0}(0):\dot L^{2,s}\to \dot L^{2,s-1}
    \qquad\text{for all}\qquad
    -\frac 12<s<\frac 32.
  \end{equation*}
  Recall also that $R_{0}(0)$ is bounded from $\dot Y^{*}$
  to $\dot X$ and $\partial R_{0}(0)$ is bounded from
  $\dot Y^{*}$ to $\dot Y$.
  Moreover from the assumption on $W,Z$ it follows that
  the corresponding multiplication operators are bounded
  operators
  \begin{equation*}
    \textstyle
    W+i (\partial \cdot Z):\dot X\to \dot L^{2,1/2+\delta},
    \qquad
    W+i (\partial \cdot Z):\dot L^{2,s-2}\to \dot L^{2,s+\delta}
    \qquad \forall s\in \mathbb{R},
  \end{equation*}
  \begin{equation*}
    Z:\dot Y\to \dot L^{2,1/2+\delta},
    \qquad
    Z:\dot L^{2,s-1}\to \dot L^{2,s+\delta}
    \qquad \forall s\in \mathbb{R}.
  \end{equation*}
  Conbining all the previous properties we deduce that
  $K(0)=(W+i \partial \cdot Z+i Z \cdot \partial)R_{0}(0)$
  is a bounded operator
  \begin{equation}\label{eq:Kbdd}
    \textstyle
    K(0):\dot Y^{*}\to \dot L^{2,1/2+\delta}
    \quad\text{and}\quad 
    K(0):\dot L^{2,s}\to \dot L^{2,s+\delta},
    \qquad
    \forall\ \frac 12<s<\frac 32.
  \end{equation}
  Since we know that $f\in \dot Y^{*}$ and that
  $f=K(0)f$, applying \eqref{eq:Kbdd} repeatedly,
  we obtain in a finite number of steps that 
  $f\in \dot L^{2,3/2}$, which in turn implies
  $v=R_{0}(0)f\in \dot L^{2,s}$ for all $s<-\frac 12$
  and
  $\partial v=\partial R_{0}(0)f\in \dot L^{2,s}$ for all 
  $s<\frac 32$. The proof is concluded.
\end{proof}

Note that $z \mapsto I-K(z)$ is trivially 
continuous (and actually holomorphic for $z\not\in \sigma(L)$).
Since $K(z)$ is compact and $I-K(z)$ is injective on $\dot Y^{*}$,
it follows from Fredholm theory that $(I-K(z))^{-1}$ is a
bounded operator for all $z\in \mathbb{C}$. However we need
a bound uniform in $z$, and to this end it is sufficient
to prove that the map $z \mapsto(I-K(z))^{-1}$
is continuous. 
This follows from a general well known result
on Fredholm operators (a proof can be found e.g. in
\cite{DAncona17-a}):

\begin{lemma}[]\label{lem:invertImK}
  Let $X_{1},X_{2}$ be two Banach spaces,
  $K_{j},K$ compact operators from $X_{1}$ to $X_{2}$,
  and assume
  $K_{j}\to K$ in the operator norm as $j\to \infty$.
  If $I-K_{j}$, $I-K$ are invertible with bounded inverses,
  then $(I-K_{j})^{-1}\to(I-K)^{-1}$ in the operator norm.
\end{lemma}

% \begin{proof}%[of ...]
%   Denote by $\|\cdot\|_{\mathscr{B}}$ the norm of bounded
%   operators $X_{1}\to X_{2}$ (or $X_{2}\to X_{1}$).
%   Let $\phi\in X_{2}$ and let
%   $c_{j}:=\|(I-K_{j})^{-1}\phi\|_{X_{1}}$.
%   If by contradiction $c_{j}\to \infty$, then
%   defining $\psi_{j}=(I-K_{j})^{-1}\phi \cdot c_{j}^{-1}$
%   and $\phi_{j}=\phi \cdot c_{j}^{-1}$ we would have
%   \begin{equation*}
%     \|\psi_{j}\|_{X_{1}}=1,
%     \qquad
%     \|\phi_{j}\|_{X_{2}}\to 0,
%     \qquad
%     \phi_{j}=(I-K_{j})\psi_{j}.
%   \end{equation*}
%   The last identity can be written
%   \begin{equation*}
%     \psi_{j}=\phi_{j}+(K_{j}-K)\psi_{j}+K \psi_{j}.
%   \end{equation*}
%   The first two terms at the right tend to 0, and the third
%   one converges, by possibly passing to a subsequence, since
%   $K$ is compact; let $\psi=\lim K \psi_{j}$. By the previous
%   identity we see that also $\psi_{j}$ converges to $\psi$
%   so that $\|\psi\|=1$ and $\psi=K \psi$, which contradicts 
%   the invertibility of $I-K$. 

%   We have thus proved that, for any $\phi\in X_{2}$, the sequence
%   $\chi_{j}:=(I-K_{j})^{-1}\phi$ is bounded in $X_{1}$.
%   Write this identity
%   \begin{equation*}
%     \chi_{j}=\phi+K \chi_{j}+(K_{j}-K)\chi_{j}
%   \end{equation*}
%   and note as before that $K \chi_{j}$ is a relatively compact
%   sequence; let $\chi$ be any one of its limit points.
%   Letting $j\to \infty$ we get $\chi=\phi+K \chi$, i.e.,
%   $(I-K_{j})^{-1}\phi\to(I-K)^{-1}\phi$.
%   Applying the uniform boundedness principle we get the claim.
% \end{proof}

We finally sum up the previous results. 
We shall need to assume that
$0$ \emph{is not a resonance}, in the following sense:

\begin{definition}[Resonance]\label{def:reson}
  We say that $0$ is a \emph{resonance} for the operator
  $L$ if there exists a nonzero
  $v\in H^{2}_{loc}(\mathbb{R}^{3}\setminus0)\cap \dot X$
  with $\partial v\in \dot Y$,
  solution of $Lv=0$ with the properties
  \begin{equation}\label{eq:resona}
    |x|^{-\frac 12-\sigma}v\in L^{2}
    \quad\text{and}\quad 
    |x|^{\frac 12-\sigma}\partial v\in L^{2}
    \qquad \forall \sigma>0.
  \end{equation}
  The function $v$ is then called a \emph{resonant state}
  at 0 for $L$.
\end{definition}

Note that in 
Lemma \ref{lem:eigv} we proved in particular that if 
$f\in \dot Y^{*}$ satisfies $f=K(0)f$, then $v=R_{0}(0)f$
is a resonant state at 0.

\begin{proposition}[]\label{pro:invert}
  Assume the operator $L$ defined in \eqref{eq:defL}
  is non negative and selfadjoint on $L^{2}$, with
  $W=W^{*}$ and $Z$ satisfying \eqref{eq:WZdelta}
  for some $\delta>0$. In addition, asssume that
  $0$ is not a resonance for $L$, in the sense of
  \eqref{eq:resona}.

  Then $I_{4}-K(z)$ is a bounded invertible operator 
  on $\dot Y^{*}$, with $(I_{4}-K(z))^{-1}$
  bounded uniformly for $z$ in bounded subsets of 
  $\overline{\mathbb{C}^{\pm}}$.
  Moreover, the resolvent operator $R(z)=(L-z)^{-1}$ 
  satisfies the estimate
  \begin{equation}\label{eq:resestim}
    \|R(z)f\|_{\dot X}
    +
    |z|^{\frac12}\|R(z)f\|_{\dot Y}
    +
    \|\partial R(z)f\|_{\dot Y}
    \le
    C(z)
    \|f\|_{\dot Y^{*}}
  \end{equation}
  for all $z\in\overline{\mathbb{C}^{\pm}}$, where $C(z)$
  is a continuous function of $z$.
\end{proposition}

\begin{proof}%[of ...]
  It is sufficient to combine
  Lemmas \ref{lem:compact},
  \ref{lem:eigvcomplex},
  \ref{lem:eigv},
  \ref{lem:invertImK}
  and apply Fredholm theory in conjuction with assumption
  \eqref{eq:resona}, to prove the claims about
  $I-K(z)$; note that \eqref{eq:WZdelta} include
  the assumptions of Lemmas \ref{lem:compact}--\ref{lem:invertImK}.
  Finally, using the representation \eqref{eq:LSE}
  and the free estimate \eqref{eq:freeest} we 
  obtain \eqref{eq:resestim}.
\end{proof}

\subsection{Proof of Theorem \ref{the:smoooestD}}
\label{sub:pro_of_thesmoo}

Squaring the operator $\mathcal{D}+V$ produces a 
non negative, selfadjoint
operator with domain $H^{2}(\mathbb{R}^{3})$, of the form
\begin{equation}\label{eq:squadirac}
  L:=
  (\mathcal{D}+V)^{2}=
  -I_{4}\Delta
  +
  V^{2}+\mathcal{D}V+V \mathcal{D}.
\end{equation}
We want to apply Propositions \ref{pro:invert} 
and \ref{pro:resestlarge} to the operator $L$.
First of all we check the 0 resonance assumption:

\begin{lemma}[]\label{lem:reson2reson}
  If $0$ is a resonance for the operator 
  $L=(\mathcal{D}+V)^{2}$, in the sense of 
  Definition \ref{def:reson}, then $0$ is a
  resonance for the operator $\mathcal{D}+V$
  in the sense of Definition \ref{def:resonDirac}.
\end{lemma}

\begin{proof}%[of ...]
  Let $v$ be the resonant state for $L$, with the properties
  listed in Definition \ref{def:reson}, and let 
  $w=(\mathcal{D}+V)v$. If $w=0$ then $v$ is a resonant state
  at 0 for $\mathcal{D}+V$ and the proof is concluded,
  thus we can assume $w$ nonzero.
  By the properties of $v$ we have directly
  $w\in H^{1}_{loc}(\mathbb{R}^{3}\setminus0)\cap \dot Y$ and
  $(\mathcal{D}+V)w=0$, so in particular $w\in L^{2}_{loc}$.
  We now prove that 
  $|x|^{\frac12-\sigma}\mathcal{D}v$
  and
  $|x|^{\frac12-\sigma}Vv$ belong to $L^{2}$;
  thus $w\in L^{2}$ which means that 
  0 is an eigenvalue of $\mathcal{D}+V$.
  The first fact is also contained in the definition of
  the resonant state $v$, while the second one is
  an immediate consequence of the property
  $|V|\lesssim|x|^{-1}$ and of the following
  generalized Hardy inequality
  \begin{equation}\label{eq:hardygen}
    \||x|^{\sigma-1/2}w\|_{L^{2}}
    \lesssim\||x|^{\sigma+1/2}\partial w\|_{L^{2}},
    \qquad
    \sigma>-1.
  \end{equation}
  The proof of \eqref{eq:hardygen} is simple:
  for a compactly supported smooth function $\phi$,
  integrate on $\mathbb{R}^{3}$ the identity
  \begin{equation*}
    \textstyle
    \partial \cdot
    \{\widehat{x}|x|^{2\sigma}|\phi|^{2}\}
    =
    (2+2\sigma)|x|^{2\sigma-1}|\phi|^{2}+
    2\Re|x|^{2\sigma}\phi \overline{\phi_{r}},
    \qquad
    \phi_{r}:= \widehat{x}\cdot \partial \phi
  \end{equation*}
  and use Cauchy--Schwartz to obtain
  \begin{equation*}
    \textstyle
    (\sigma+1)\int_{B_{R}}|x|^{2\sigma-1}|\phi|^{2}dx
    \le
    2(\int_{B_{R}}|x|^{2\sigma-1}|\phi|^{2}dx)^{1/2}
    (\int_{B_{R}}|x|^{2\sigma+1}|\phi_{r}|^{2}dx)^{1/2}.
    \qedhere
  \end{equation*}
\end{proof}

We next check that $L$ satisfies assumption \eqref{eq:WZdelta}.
Following \eqref{eq:squadirac} we must choose
\begin{equation*}
  W:=-V^{2},
  \qquad
  Z:=\alpha V.
\end{equation*}
It is easy to checl that conditions
\eqref{eq:WZdelta} are implied by
\begin{equation}\label{eq:secondass}
  |V|^{2}+|\mathcal{D}V|\le \frac{C}{|x|^{2+\delta}},
  \qquad
  |x|^{2}(|V|^{2}+|\mathcal{D}V|)\in \ell^{1}L^{\infty}
\end{equation} 
for some $\delta>0$
(compare with Condition (V)).
Note that the first condition is
effective for large $x$ while the second one restricts the
singularity at 0 of the potential.
Then we are in position to apply Proposition \ref{pro:invert}
and we obtain that the resolvent operator
$R(z)=(L-z)^{-1}$ satisfies the estimate
\begin{equation*}
  \|R(z)f\|_{\dot X}
  +
  |z|^{\frac12}\|R(z)f\|_{\dot Y}
  +
  \|\partial R(z)f\|_{\dot Y}
  \le
  C(z)\|f\|_{\dot Y^{*}}
\end{equation*}
for all $z\in \overline{\mathbb{C}^{\pm}}$,
with a constant $C(z)$ depending continuously on $z$.

Next, in order to apply Proposition \ref{pro:resestlarge},
using the decomposition $V=\alpha \cdot A+A_{0}\beta+V_{0}$,
we may rewrite $L$ in the form
\begin{equation}\label{eq:squaDirac2}
  \textstyle
  L:=
  -
  I_{4}\Delta_{A}
  -
  i\{V_{0},\alpha\}\partial^{A}v
  -
  i\sum_{j<k}  B_{jk}\alpha_{j}\alpha_{k}
  +
  \mathcal{D}(V_{0}+A_{0}\beta)
  +
  (V_{0}+A_{0}\beta)^{2},
\end{equation}
where
\begin{equation*}
  \{V_{0},\alpha\}=
  V_{0}\alpha+\alpha V_{0},
  \qquad
  B_{jk}=\partial_{j}A_{k}-\partial_{k}A_{j},
  \quad
  j,k=1,2,3.
\end{equation*}
By comparing with \eqref{eq:system}, we choose now
\begin{equation*}
  Z:=\{V_{0},\alpha\},
  \qquad
  -W:=
  -
  i\sum_{j<k}  B_{jk}\alpha_{j}\alpha_{k}
  +
  \mathcal{D}(V_{0}+A_{0}\beta)
  +
  (V_{0}+A_{0}\beta)^{2}
\end{equation*}
and we verify that assumption \eqref{eq:asscoeff} is
satisfied as soon as we impose on the coefficients, besides
\eqref{eq:secondass}, the conditions
\begin{equation}\label{eq:firstass}
  |x||B|+|x|^{2}|\mathcal{D}V_{0}| \in \ell^{1}L^{\infty},
  \qquad
  \||x|V_{0}\|_{\ell^{1}L^{\infty}}< \sigma_{0}
\end{equation}
with $\sigma_{0}>0$ as in Proposition \ref{pro:resestlarge}
(compare with Consition (V)).
From \eqref{eq:firstass} it follows directly that
$\||x|Z\|_{\ell^{1}L^{\infty}}<\sigma_{0}$ and
$|x| \widehat{B}\in \ell^{1}L^{\infty}$ as required.
Next we define
\begin{equation*}
  W_{S}:=
  -
  \left[
    \mathcal{D}(V_{0}+A_{0}\beta)+(V_{0}+A_{0}\beta)^{2}
  \right]
  \one{|x|<R}
\end{equation*}
where $\one{|x|<R}$ is the characteristic function of
$\{|x|<R\}$, and we remark that
\begin{equation*}
  |x|^{2}\mathcal{D}(V_{0}+A_{0}\beta)\in \ell^{1}L^{\infty}
\end{equation*}
since both $V$ and $V_{0}$ satisfy a similar assumpion
(and hence also $A,A_{0}$, in view of the linear independence
of Dirac matrices); on the other hand
\begin{equation*}
  |x|V_{0}\in \ell^{1}L^{\infty}\hookrightarrow \ell^{2}L^{\infty}
  \quad\implies\quad
  |x|^{2}V_{0}^{2}\in \ell^{1}L^{\infty}
\end{equation*}
and since $|x|^{2}V^{2}\in \ell^{1}L^{\infty}$ we have also
\begin{equation*}
  |x|^{2}A_{0}^{2}\in \ell^{1}L^{\infty}.
\end{equation*}
All this implies that $|x|^{2}W_{S}\in \ell^{1}L^{\infty}$,
and hence
\begin{equation*}
  \lim_{R\to0}
  \||x|^{2}W_{S}\|_{\ell^{1}L^{\infty}}\to0.
\end{equation*}
Picking $R$ sufficiently small we see that $W_{S}$ satisfies
\eqref{eq:asscoeff}. It remains to check that
$W_{L}:=W-W_{S}$ satisfies $|x|W_{L}\in \ell^{1}L^{\infty}$,
and this follows from assumption \eqref{eq:firstass} on $B$
and from the previous estimates on $V_{0},A_{0}$
(thanks to the cutoff $1-\one{|x|<R}$ vanishing near 0).

Thus all the assumptions of 
Proposition \ref{pro:resestlarge} are satisfied
and we have
\begin{equation*}
  \|R(z)f\|_{\dot X}
  +
  |z|^{1/2}\|R(z)f\|_{\dot X}
  +\|\partial^{A} R(z)f\|_{\dot Y}
  \le
  C\|f\|_{\dot Y^{*}}
\end{equation*}
for all $z$ large enough in the strip $|\Im z|\le1$.
Taking into account Remark \ref{rem:deAtode} in order
to replace $\partial^{A}$ with $\partial$, and the previous
estimate for small $z$, we conclude that  the estimate
\begin{equation}\label{eq:resestschro}
  \|R(z)f\|_{\dot X}
  +  |z|^{1/2}\|R(z)f\|_{\dot X}
  +\|\partial R(z)f\|_{\dot Y}
  \le
  C\|f\|_{\dot Y^{*}}
\end{equation}
holds for all $z$ in the strip $|\Im z|\le1$, with a
constant independent of $z$, provided
\eqref{eq:firstass}, \eqref{eq:secondass} hold
and $A\in \ell^{\infty}L^{3}$.

Since $|V|\lesssim|x|^{-1}$
and $\||x|^{-1}v\|_{\dot Y}\lesssim\|v\|_{\dot X}$, this implies
\begin{equation}\label{eq:estDV}
  \|(\mathcal{D}+V)R(z)f\|_{\dot Y}
  \le
  C\|f\|_{\dot Y^{*}}
\end{equation}
uniformly in the strip $|\Im z|\le1$.
Moreover, for any positive function $\rho\in\ell^{2}L^{\infty}$,
using the inequalities
\begin{equation*}
  \|\rho f\|_{L^{2}}\le
  \|\rho\|_{\ell^{2}L^{\infty}}\|f\|_{\ell^{\infty}L^{2}},
  \qquad
  \|f\|_{\ell^{1} L^{2}}\le
  \|\rho\|_{\ell^{2}L^{\infty}}\|\rho^{-1} f\|_{L^{2}},  
\end{equation*}
we deduce from \eqref{eq:estDV} the estimate
\begin{equation}\label{eq:weighted}
  \|\rho|x|^{-1/2}(\mathcal{D}+V)R(z)
    |x|^{-1/2}\rho f\|_{L^{2}}
    \le C\|\rho\|_{\ell^{2}L^{\infty}}^{2}\|f\|_{L^{2}}.
\end{equation}

We now introduce the spectral projections $P_{+},P_{-}$ 
defined as
\begin{equation*}
  \textstyle
  P_{+}=\int_{0}^{+\infty}dE_{\lambda},
  \qquad
  P_{-}=\int_{-\infty}^{0}dE_{\lambda}
\end{equation*}
where $dE_{\lambda}$ is the spectral measure
of the selfadjoint operator $\mathcal{D}+V$. 
We decompose $L^{2}$ accordingly as
\begin{equation*}
  L^{2}=\mathcal{H}_{+}\oplus \mathcal{H}_{-},
  \qquad
  \mathcal{H}_{\pm}=P_{\pm}L^{2}
\end{equation*}
and we denote with $L_{\pm}=LP_{\pm}$ the parts of $L$
in $\mathcal{H}_{\pm}$. Note that
\begin{equation*}
  (\mathcal{D}+V)(L-z)^{-1}P_{\pm}=
  \pm L_{\pm}^{1/2}(L_{\pm}-z)^{-1}.
\end{equation*}
Then for $f\in \mathcal{H}_{\pm}$
we have from \eqref{eq:weighted}
\begin{equation}\label{eq:estLpm}
  \|\rho|x|^{-1/2}L_{\pm}^{1/4}
    (L_{\pm}-z)^{-1}
    \rho|x|^{-1/2}L_{\pm}^{1/4}
    f\|_{L^{2}}
  \le
  C\|\rho\|_{\ell^{2}L^{\infty}}^{2}\|f\|_{L^{2}}.
\end{equation}
This means that 
the operator $\rho|x|^{-1/2}L_{+}^{1/4}$
(resp. the operator $\rho|x|^{-1/2}L_{-}^{1/4}$)
is \emph{supersmoothing} for the selfadjoint operator
$L_{+}$ on the Hilbert space $\mathcal{H}_{+}$
(resp. for $L_{-}$ on $\mathcal{H}_{-}$)
in the sense of Kato--Yajima
\cite{KatoYajima89-a};
see \cite{DAncona15-a} for a detailed account of the theory.
By the Kato smoothing theory, this implies the following
smoothing estimate for the Schr\"{o}dinger flow $e^{itL_{\pm}}$
\begin{equation*}
  \|\rho|x|^{-1/2}L_{\pm}^{1/4}
  e^{itL_{\pm}}f\|_{L^{2}_{t}L^{2}_{x}}
  \lesssim
  \|\rho\|_{\ell^{2}L^{\infty}}
  \|f\|_{L^{2}},
  \qquad
  f\in \mathcal{H}_{\pm}
\end{equation*}
and an analogous nonhomogeneous estimate
for $\int_{0}^{t}e^{i(t-s)L_{\pm}}F(s)ds$.
However, by Theorem 2.4 in \cite{DAncona15-a},
a smoothing estimate holds also for the wave flows
$e^{itL_{\pm}^{1/2}}$, with a $L_{\pm}^{1/4}$ derivative loss:
\begin{equation*}
  \|\rho|x|^{-1/2}L_{\pm}^{1/4}e^{itL_{\pm}^{1/2}}f
  \|_{L^{2}_{t}L^{2}_{x}}
  \lesssim
  \|\rho\|_{\ell^{2}L^{\infty}}
  \|L_{\pm}^{1/4}f\|_{L^{2}},
  \qquad
  f\in \mathcal{H}_{\pm}
\end{equation*}
(and similarly for the nonhomogeneous flows
$\int_{0}^{t}e^{i(t-t')L^{1/2}_{\pm}}F(t')dt'$)
so that we have proved
\begin{equation*}
  \|\rho|x|^{-1/2} e^{itL_{\pm}^{1/2}}f\|_{L^{2}_{t}L^{2}_{x}}
  \lesssim
  \|\rho\|_{\ell^{2}L^{\infty}}
  \|f\|_{L^{2}},
  \qquad
  f\in \mathcal{H}_{\pm}.
\end{equation*}
Since $L_{\pm}^{1/2}P_{\pm}=\pm(\mathcal{D}+V)P_{\pm}$,
we arrive at
\begin{equation}\label{eq:finalsmoo}
  \|\rho|x|^{-1/2} e^{it(\mathcal{D}+V)}
  P_{\pm}f\|_{L^{2}_{t}L^{2}_{x}}
  \lesssim
  \|\rho\|_{\ell^{2}L^{\infty}}
  \|f\|_{L^{2}},
  \qquad
  f\in \mathcal{H}_{\pm},
\end{equation}
and summing over $\pm$ we obtain \eqref{eq:smooDiracV}.
The same argument gives the nonhomogeneous estimate
\eqref{eq:smooDiracVnh}. 

Finally, let $u=e^{it(\mathcal{D}+V)}f$,
and let $u_{j}=\partial_{j}u$, $f_{j}=\partial_{j}f$,
$V_{j}=\partial_{j}V$; by differentiating the equation
$iu_{t}+(\mathcal{D}+V)u=0$ we have
\begin{equation*}
  i\partial_{t}u_{j}+
  (\mathcal{D}+V)u_{j}=
  -V_{j}u,
  \qquad
  u_{j}(0)=f_{j}
\end{equation*}
so that
\begin{equation*}
  \textstyle
  u_{j}=e^{i(\mathcal{D}+V)}f_{j}
  +i
  \int_{0}^{t}e^{i(t-t')(\mathcal{D}+V)}
  V_{j}udt'
\end{equation*}
and by \eqref{eq:smooDiracV}, \eqref{eq:smooDiracVnh}
\begin{equation*}
  \|\rho|x|^{-1/2}u_{j}\|_{L_t^{2}L_x^{2}}
  \lesssim
  \|\rho\|_{\ell^{2}L^{\infty}}
  \|f_{j}\|_{L^{2}}
  +
  \|\rho\|_{\ell^{2}L^{\infty}}^{2}
  \|\rho^{-1}|x|^{1/2}V_{j}u\|_{L_t^{2}L_x^{2}}.
\end{equation*}
Then we can write
\begin{equation*}
  \|\rho^{-1}|x|^{1/2}V_{j}u\|_{L_t^{2}L_x^{2}}
  \le
  \|\rho^{-2}|x|V_{j}\|_{L^{\infty}}
  \|\rho|x|^{-1/2}u\|_{L_t^{2}L_x^{2}}
  \lesssim
  \|\rho\|_{\ell^{2}L^{\infty}}
  \|\rho^{-2}|x|V_{j}\|_{L^{\infty}}
  \|f\|_{L^{2}}
\end{equation*}
again by \eqref{eq:smooDiracVnh},
and in conclusion
\begin{equation*}
  \|\rho|x|^{-1/2}\partial u\|_{L_t^{2}L_x^{2}}
  \lesssim
  \|\rho\|_{\ell^{2}L^{\infty}}
  (1+\|\rho\|_{\ell^{2}L^{\infty}}^{2}
  \|\rho^{-2}|x|V_{j}\|_{L^{\infty}})
  \|f\|_{H^{1}}
\end{equation*}
and this gives \eqref{eq:smooDiracVD}.

\subsection{Proof of Corollary \ref{cor:smooang}}

The scalar operator 
$\Lambda_{\omega}=(1-\Delta_{\mathbb{S}^{2}})^{1/2}$,
used to define the Sobolev norms on the sphere,
is not convenient when working with the Dirac equation since
it does not commute with $\mathcal{D}$. We shall use
instead the \emph{spin--orbit operator} $K$, defined 
on $L^{2}(\mathbb{R}^{3})^{4}$
as
\begin{equation*}
  K:=
  \beta(2 S \cdot \Omega+1)
\end{equation*}
where $\Omega=(\Omega_{1},\Omega_{2},\Omega_{3})$ 
are the tangential vector fields to $\mathbb{S}^{2}$
\begin{equation*}
  \Omega=x \wedge \partial
\end{equation*}
while $S=(S_{1},S_{2},S_{3})=-\frac i4 \alpha \wedge \alpha$ 
are the constant matrices
\begin{equation*}
  \textstyle
  S_{j}=-\frac i2 \alpha_{k}\alpha_{\ell},
  \qquad
  (j,k,\ell)
  \ \text{a cyclic permutation of}\ (1,2,3).
\end{equation*}

To describe the action of the Dirac operator it is necessary
to recall the \emph{partial wave decomposition} of
$L^{2}(\mathbb{S}^{2})^{4}$. 
See Section 4.6 of \cite{Thaller92-a} for a complete account.
Let
$Y^{m}_{\ell}$, $\ell=0,1,2,\dots$, $m=-\ell,-\ell+1,\dots,\ell$, 
the usual spherical harmonics on $\mathbb{S}^{2}$, which are an
orthonormal basis of $L^{2}(\mathbb{S}^{2})$; then 
an orthonormal basis of $L^{2}(\mathbb{S}^{2})^{4}$
is given by the family of functions
\begin{equation}\label{eq:Phimk}
  \Phi^{\pm}_{m_{j},k_{j}},
  \qquad
  \textstyle
  j=\frac12,\frac32,\frac52,\dots\qquad
  m_{j}=-j,-j+1,\dots,j,\qquad
  k_{j}=\pm(j+\frac12)
\end{equation}
defined as follows: when $k_{j}=j+1/2$ we have
\begin{equation*}
  \textstyle
  \Phi^{+}_{m_{j},k_{j}}=
   \frac{i}{\sqrt{2j+2}}
  \begin{pmatrix}
    \sqrt{j+1-m_{j}}\ \ Y^{m_{j}-1/2}_{k_{j}} \\
    -\sqrt{j+1+m_{j}}\ \ Y^{m_{j}+1/2}_{k_{j}}\\
    0 \\
    0
  \end{pmatrix}
  \quad
  \textstyle
  \Phi^{-}_{m_{j},k_{j}}=
   \frac{1}{\sqrt{2j}}
  \begin{pmatrix}
    0 \\
    0 \\
    \sqrt{j+m_{j}}\ \ Y^{m_{j}-1/2}_{k_{j}-1} \\
    \sqrt{j-m_{j}}\ \ Y^{m_{j}+1/2}_{k_{j}-1}
  \end{pmatrix}
\end{equation*}
while when $k_{j}=-(j+1/2)$ we have
\begin{equation*}
  \textstyle
  \Phi^{+}_{m_{j},k_{j}}=
   \frac{i}{\sqrt{2j}}
  \begin{pmatrix}
    \sqrt{j+m_{j}}\ \ Y^{m_{j}-1/2}_{1-k_{j}} \\
    \sqrt{j-m_{j}}\ \ Y^{m_{j}+1/2}_{1-k_{j}}\\
    0 \\
    0
  \end{pmatrix}
  \quad
  \textstyle
  \Phi^{-}_{m_{j},k_{j}}=
   \frac{1}{\sqrt{2j+2}}
  \begin{pmatrix}
    0 \\
    0 \\
    \sqrt{j+1-m_{j}}\ \ Y^{m_{j}-1/2}_{-k_{j}} \\
    -\sqrt{j+1+m_{j}}\ \ Y^{m_{j}+1/2}_{-k_{j}}
  \end{pmatrix}.
\end{equation*}
For each choice of $j,m_{j},k_{j}$ as in \eqref{eq:Phimk},
the couple $\{\Phi_{m_{j};k_{j}}^{+},\Phi_{m_{j};k_{j}}^{-}\}$
generates a 2D subspace
$H_{m_{j},k_{j}}$ of $L^{2}(\mathbb{S}^{2})^{4}$,
and we have the natural decomposition
\begin{equation*}
  L^{2}(\mathbb{R}^{3})^{4}\simeq
  \bigoplus_{j=\frac12,\frac32,\dots}^{\infty}
  \bigoplus_{m_{j}=-j}^{j}
  \bigoplus_{\genfrac{}{}{0pt}{2}{k_{j}=}{\pm(j+1/2)}}
  L^{2}(0,+\infty;dr)
  \otimes
  H_{m_{j},k_{j}}.
\end{equation*}
The isomorphism is expressed by the explicit expansion
\begin{equation}\label{eq:expans}
  \Psi(x)=\sum 
    \frac1r
    \psi^{+}_{m_{j},k_{j}}(r)
    \Phi^{+}_{m_{j},k_{j}}+
    \frac1r
    \psi^{-}_{m_{j},k_{j}}(r)
    \Phi^{-}_{m_{j},k_{j}}
\end{equation}
with
\begin{equation}\label{eq:L2exp}
  \|\Psi\|^{2}_{L^{2}}=
  \sum\int_{0}^{\infty}[
    |\psi^{+}_{m_{j},k_{j}}|^{2}+|\psi^{-}_{m_{j},k_{j}}|^{2}
  ]dr.
\end{equation}
Notice also that
\begin{equation}\label{eq:L2expS}
  \|\Psi\|^{2}_{L^{2}(\mathbb{S}^{2})}=
    \sum
    \frac{1}{r^{2}}
    |\psi^{+}_{m_{j},k_{j}}|^{2}+
    \frac{1}{r^{2}}|
    \psi^{-}_{m_{j},k_{j}}|^{2}.
\end{equation}
Each summand
$L^{2}(0,+\infty; dr)\otimes H_{m_{j},k_{j}}$ 
is an eigenspace of the
Dirac operator $\mathcal{D}=i^{-1}\sum \alpha_{j}\partial_{j}$
and the action of $\mathcal{D}$ can be written, 
in terms of the expansion \eqref{eq:expans}, as
\begin{equation*}
  \mathcal{D}\Psi=\sum 
    \left(-\frac{d}{dr}\psi^{-}_{m_{j},k_{j}}+
      \frac{k_{j}}{r}\psi^{-}_{m_{j},k_{j}}
    \right)
    \frac{\Phi^{+}_{m_{j},k_{j}}}{r}+
    \left(\frac{d}{dr}\psi^{+}_{m_{j},k_{j}}+
        \frac{k_{j}}{r}\psi^{+}_{m_{j},k_{j}}
    \right)
    \frac{\Phi^{-}_{m_{j},k_{j}}}{r}.
\end{equation*}
Note that the $\Phi^{\pm}_{m_{j}mk_{j}}$ 
are eigenvectors for $\Lambda_{\omega}$
but with different eigenvalues (satisfying $\simeq j$),
while $\mathcal{D}$ swaps them,
hence $\mathcal{D}$ amd $\Lambda_{\omega}$ do not commute.
On the other hand, the spin--orbit operator $K$ satisfies
\begin{equation*}
  K \Phi^{\pm}_{m_{j},k_{j}}
  =
  -k_{j}\Phi^{\pm}_{m_{j},k_{j}}.
\end{equation*}
Since $k_{j}\simeq\pm j$, we have obviously
\begin{equation}\label{eq:equivKL}
  \|Kv\|_{L^{2}(\mathbb{S}^{2})}
  \simeq
  \|\Lambda_{\omega} v\|_{L^{2}(\mathbb{S}^{2})}
\end{equation}
and more generally, if we define $|K|^{s}$ via
\begin{equation}\label{eq:defKs}
  |K|^{s} \Phi^{\pm}_{m_{j},k_{j}}
  =
  k_{j}^{s}\Phi^{\pm}_{m_{j},k_{j}}.
\end{equation}
we have also
\begin{equation*}
  \||K|^{s}v\|_{L^{2}(\mathbb{S}^{2})}
  \simeq
  \|\Lambda_{\omega}^{s} v\|_{L^{2}(\mathbb{S}^{2})}.
\end{equation*}
Thus the differential operator $K$ can replace 
$\Lambda_{\omega}$ 
to measure angular regularity of functions.
Moreover $K$ commutes with the Dirac matrix $\beta$:
\begin{equation*}
  [K,\beta]=0
\end{equation*}
and as a consequence, the commutator $[K,A_{0}\beta]$ is a
bounded operator on $L^{2}(\mathbb{S}^{2})$:
\begin{equation}\label{eq:commbdd}
  \|[K,A_{0}\beta]v\|_{L^{2}(\mathbb{S}^{2})}
  \lesssim
  \|\Omega A_{0}\|_{L^{\infty}(\mathbb{S}^{2})}
  \|v\|_{L^{2}(\mathbb{S}^{2})}.
\end{equation}

We turn now to the proof of Corollary \ref{cor:smooang}.
Assume first $V_{0}=0$ i.e. $V=A_{0}\beta$ only.
Then by applying $K$ to the equation we get
\begin{equation*}
  iu_{t}+(\mathcal{D}+V)u=F
  \quad\implies\quad
  i(Ku)_{t}+(\mathcal{D}+V)(Ku)=
  KF+
  [K,A_{0}\beta]u.
\end{equation*}
By estimates \eqref{eq:smooDiracV}--\eqref{eq:smooDiracVnh} we have
then
\begin{equation*}
  \|\rho|x|^{-1/2}Ku\|_{L_t^{2}L_x^{2}}
  \lesssim
  \|Ku(0)\|_{L^{2}}+
  \|\rho^{-1}|x|^{1/2}([K,A_{0}\beta]u+KF)\|_{L_t^{2}L_x^{2}}
\end{equation*}
and using \eqref{eq:commbdd},
\eqref{eq:assangA0}
and the estimates \eqref{eq:smooDiracV}, \eqref{eq:smooDiracVnh}
already proved, we obtain
\begin{equation*}
  \|\rho|x|^{-1/2}Ku\|_{L_t^{2}L_x^{2}}
  \lesssim
  \|u(0)\|_{L^{2}}+
  \|Ku(0)\|_{L^{2}}+
  \|\rho^{-1}|x|^{1/2}F\|_{L_t^{2}L_x^{2}}+
  \|\rho^{-1}|x|^{1/2}KF\|_{L_t^{2}L_x^{2}}.
\end{equation*}
Using the equivalence \eqref{eq:equivKL} on the
sphere, we obtain 
\eqref{eq:smooDiracVang}, \eqref{eq:smooDiracVangnh}
for $s=1$.
By interpolation with the case $s=0$, we have proved
\eqref{eq:smooDiracVang}, \eqref{eq:smooDiracVangnh}
for all $0\le s\le1$ under the
additional assumption $V_{0}=0$. The same argument gives
the estimate in the range $1\le s\le 2$, if $V_{0}=0$.

Assume now $V_{0}\neq0$. We have
\begin{equation*}
  iu_{t}+(\mathcal{D}+V)u=F
  \quad\implies\quad
  iu_{t}+(\mathcal{D}+A_{0}\beta)u=F-V_{0}u
\end{equation*}
and by the previous part of the proof
\begin{equation*}
  \|\rho|x|^{-1/2}\Lambda^{s}_{\omega} u\|_{L_t^{2}L_x^{2}}
  \lesssim
  \|\Lambda^{s}_{\omega}u(0)\|_{L^{2}}
  +
  \|\rho^{-1}|x|^{1/2}\Lambda^{s}_{\omega}(F-V_{0}u)\|_{L_t^{2}L_x^{2}}.
\end{equation*}
If $s>1$ we can use the product rule
\begin{equation} \label{eq:est:prod}
\| \Lambda^s_{\omega} (fg) \|_{L_{\omega}^2 (\Si^2)} \lesssim \| \Lambda^s_{\omega} f \|_{L_{\omega}^2 (\Si^2)} \| \Lambda^s_{\omega} g \|_{L_{\omega}^2 (\Si^2)}
\end{equation}
(see (4.9) in \cite{CacDAn}).
Then we have
\begin{equation*}
  \|\Lambda^{s}_{\omega}(V_{0}u)\|_{L^{2}(\mathbb{S}^{2})}
  \lesssim
  \|\Lambda^{s}_{\omega}V_{0}\|_{L^{2}(\mathbb{S}^{2})}
  \|\Lambda^{s}_{\omega}u\|_{L^{2}(\mathbb{S}^{2})}
  \lesssim
  \epsilon \rho^{2}|x|^{-1}
  \|\Lambda^{s}_{\omega}u\|_{L^{2}(\mathbb{S}^{2})}
\end{equation*}
where we used assumption \eqref{eq:assang}, and
if $\epsilon$ is sufficiently small the resulting term
can be absorbed at the left hand side, proving
\eqref{eq:smooDiracVang}, \eqref{eq:smooDiracVangnh}
also for nonzero $V_{0}$.

To prove the last estimate \eqref{eq:smooDiracVangd}
it is sufficient to differentiate the equation
(with $F=0$)
\begin{equation*}
  i(\partial _{j}u)_{t}
  +
  (\mathcal{D}+V)(\partial _{j}u)=
  -
  (\partial _{j}V)u
\end{equation*}
and apply \eqref{eq:smooDiracVang}, \eqref{eq:smooDiracVangnh},
using again the product estimate and assumption \eqref{eq:assang}
as above in order to estimate the term
$(\partial _{j}V)u$, and then estimate \eqref{eq:smooDiracVang}
already proved.

\section{Endpoint Strichartz estimates} \label{sec:str_est}

Strichartz estimates for the free Dirac equation on $\mathbb{R}^{3}$
take the form
\begin{equation}\label{eq:strfree}
  \|e^{it \mathcal{D}}f\|
  _{L^{p}_{t}L^{q}}
  \lesssim
  \||D|^{\frac2p}f\|_{L^{2}},
\end{equation}
\begin{equation}\label{eq:strfreenh}
  \textstyle
  \|
  |D|^{-\frac2p} 
  \int_{0}^{t}
  e^{i(t-t')\mathcal{D}}Fdt' \|_{L^{p}_{t}L^{q}}
  \lesssim
  \||D|^{\frac2{\widetilde{p}}}F\|
  _{L^{\widetilde{p}'}_{t}L^{\widetilde{q}'}},
\end{equation}
where $(p,q)$ and $(\widetilde{p},\widetilde{q})$
are unrelated couples of admissible indices, i.e., satisfying
\begin{equation*}
  \frac1p+\frac1q=\frac12,
  \qquad
  2\le q<\infty,
  \qquad
  \infty\ge p>2.
\end{equation*}
The estimates fail at the so--called
\emph{endpoint} $(p,q)=(2,\infty)$,
however the following replacement is true:
\begin{equation*}
  \|e^{it \mathcal{D}}f\|_{L^{2}_{t}L^{\infty}L^{2}}
  \lesssim
  \|f\|_{\dot H^{1}}.
\end{equation*}
Moreover, we have
the mixed Strichartz--smoothing endpoint estimate
\begin{equation*}
  \textstyle
  \|
  \int_{0}^{t}
  e^{i(t-t')\mathcal{D}}Fdt' \|_{L^{2}_{t}L^{\infty}L^{2}}
  \lesssim
  \|\bra{x}^{\frac12+}|D|F\|_{L^{2}_{t}L^{2}}.
\end{equation*}
Both estimates are proved in \cite{CacDAn}.
Actually, by a minor modification in the arguments
of \cite{CacDAn}, we can prove the following:

\begin{proposition}[]\label{pro:strfreesmoonew}
  Let $\rho\in \ell^{2}L^{\infty}$, $\rho>0$,
  $\rho$ radially symmetric.
  For all $s\ge0$, the flow $e^{it \mathcal{D}}$ satisfies
  the estimates
  \begin{equation}\label{eq:strfreeend}
    \|
    \Lambda^{s}_{\omega}
    e^{it \mathcal{D}}f\|_{L^{2}_{t}L^{\infty}L^{2}}
    \lesssim
    \|
    \Lambda^{s}_{\omega}
    f\|_{\dot H^{1}}
  \end{equation}
  and
  \begin{equation}\label{eq:strfreesmoonew}
    \textstyle
    \|
    \Lambda^{s}_{\omega}
    \int_{0}^{t}
    e^{i(t-t')\mathcal{D}}Fdt' \|_{L^{2}_{t}L^{\infty}L^{2}}
    \lesssim
    \|\rho^{-1}|x|^{\frac12}|D|
    \Lambda^{s}_{\omega}
    F\|_{L^{2}_{t}L^{2}}.
  \end{equation}
\end{proposition}

\begin{proof}%[of ...]
  The first estimate is precisely (2.36) of 
  Corollary 2.4 in \cite{CacDAn}.
  In order to prove \eqref{eq:strfreesmoonew},
  we argue exactly as in the proofs of Theorem 2.3 and
  Corollary 2.4 in \cite{CacDAn}, expanding the flow
  in spherical harmonics. The only modification
  is to replace the estimate after formula (2.30) in that paper
  with the following one:
  \begin{equation*}
    \textstyle
    \int_{0}^{t}|\widehat{G}^{\ell}_{k}|ds
    \le
    \int_{-\infty}^{+\infty}
    w(\lambda+t-s)^{-1}
    w(\lambda+t-s)
    |\widehat{G}^{\ell}_{k}(s,\lambda+t-s)|ds
    \le
    \|w(r)^{-1}\|_{L^{2}}
    Q_{k}^{\ell}(\lambda+t)
  \end{equation*}
  where the weight $w$ is now $w(r)=|r|^{1/2}\rho(|r|)^{-1}$ 
  instead of $w(r)=\bra{r}^{\frac12+}$, and
  \begin{equation*}
    \textstyle
    Q_{k}^{\ell}(\mu)
    :=
    \left(\int_{-\infty}^{+\infty}
    w(\mu-s)^{2}
    |\widehat{G}^{\ell}_{k}(s,\mu-s)|^{2}
    \right)^{\frac12}.
  \end{equation*}
  Since we have
  \begin{equation*}
    \|w(r)^{-1}\|_{L^{2}(\mathbb{R})}
    =
    \||r|^{-\frac12}\rho(|r|)\|_{L^{2}(\mathbb{R})}
    \le
    \|\rho\|_{\ell^{2}L^{\infty}}
    \||r|^{-\frac12}\|_{\ell^{\infty}L^{2}}
    <\infty,
  \end{equation*}
  this implies
  \begin{equation*}
    \textstyle
    \int_{0}^{t}|\widehat{G}^{\ell}_{k}(s,t-s+\lambda)|ds
    \lesssim
    Q^{\ell}_{k}(\lambda+t)
  \end{equation*}
  as in \cite{CacDAn}. The rest of the proof is unchanged.
\end{proof}

With the help of \eqref{eq:strfreeend}, \eqref{eq:strfreesmoonew}
we can deduce from the smoothing estimates of Theorem 
\ref{the:smoooestD} the endpoint Strichartz estimates for the 
perturbed flow:

\begin{theorem}[]\label{the:strichDV}
  Let $\rho\in \ell^{2}L^{\infty}$, radially symmetric,
  with $\rho^{-2}|x|\in A_{2}$. 
  Assume Condition (V) holds with $\sigma$ small enough.
  If in addition we assume
  \begin{equation}\label{eq:assV2}
    \rho^{-2}|x|(|V|+|\partial V|)\in L^{\infty},
  \end{equation}
  then the perturbed flow satisfies
  \begin{equation}\label{eq:strichDV}
    \|
    e^{it(\mathcal{D}+V)}f\|_{L^{2}_{t}L^{\infty}L^{2}}
    \lesssim
    \|
    f\|_{H^{1}}.
  \end{equation}
  On the other hand, if $V$ has the special form
  \begin{equation*}
    V=A_{0}\beta+V_{0}
  \end{equation*}
  and satisfies (besides Condition (V))
  the assumptions \eqref{eq:assang},
  \eqref{eq:assangA0} for some $s>1$,
  then we have
  \begin{equation}\label{eq:strichDVla}
    \|
    \Lambda^{s}_{\omega}
    e^{it(\mathcal{D}+V)}f\|_{L^{2}_{t}L^{\infty}L^{2}}
    +    \|
    \Lambda^{s}_{\omega}
    e^{it(\mathcal{D}+V)}f\|_{L^{\infty}_{t}H^{1}}
    \lesssim
    \|
    \Lambda^{s}_{\omega}
    f\|_{H^{1}}.
  \end{equation}
\end{theorem}

\begin{proof}%[of ...]
  By Duhamel's formula we can write
  \begin{equation}\label{eq:duha}
    \textstyle
    e^{it(\mathcal{D}+V)}f
    =
    e^{it \mathcal{D}}f
    -
    i\int_{0}^{t}
    e^{i(t-t')\mathcal{D}}(Vu)dt',
  \end{equation}
  where $u= e^{it (\D +V)} f$.
  By \eqref{eq:strfreeend}, \eqref{eq:strfreesmoonew} we get
  \begin{equation*}
    \|
    \Lambda^{s}_{\omega}
    e^{it(\mathcal{D}+V)}f\|_{L^{2}_{t}L^{\infty}L^{2}}
    \lesssim
    \|
    \Lambda^{s}_{\omega}
    f\|_{\dot H^{1}}
    +
    \|\rho^{-1}|x|^{1/2}
    \Lambda^{s}_{\omega}
    |D|(Vu)\|_{L^{2}_{t}L^{2}}.
  \end{equation*}
  Since $\rho^{-2}|x|\in A_{2}$,
  we can replace $|D|$ by $\partial$ in the last term:
  \begin{equation*}
    \|\rho^{-1}|x|^{1/2}
    \Lambda^{s}_{\omega}
    |D|(Vu)\|_{L^{2}_{t}L^{2}}
    \simeq
    \|\rho^{-1}|x|^{1/2}
    \Lambda^{s}_{\omega}
    ((\partial V)u+V(\partial u))\|
    _{L^{2}_{t}L^{2}}.
  \end{equation*}
  In the case $s=0$, we continue the estimate as follows
  \begin{equation*}
    \lesssim
    \|\rho^{-2}|x|(|V|+|\partial V|)\|_{L^{\infty}}
    \|\rho|x|^{-1/2}(|u|+|\partial u|)\|_{L^{2}_{t}L^{2}}
  \end{equation*}
  and using the smoothing estimates of Theorem \ref{the:smoooestD}
  and assumption \eqref{eq:assV2}, we obtain \eqref{eq:strichDV}.
  If instead $s>1$, we estimate as follows
  \begin{equation*}
    \lesssim
    \|\rho^{-2}|x|
    (| \Lambda^{s}_{\omega} \partial V|
    +
    | \Lambda^{s}_{\omega}V|)\|_{L^{\infty}L^{2}}
    \|\rho|x|^{-1/2}(| \Lambda^{s}_{\omega} u|
    +
    | \Lambda^{s}_{\omega} \partial u|)\|_{L^{2}_{t}L^{2}}
  \end{equation*}
  thanks to the product rule \eqref{eq:est:prod},
  and using the estimates of 
  Corollary \ref{cor:smooang} we obtain the first part of
  \eqref{eq:strichDVla}.

  It remains to prove the second part of \eqref{eq:strichDVla},
  i.e., the energy estimate with angular regularity.
  First of all we note that the untruncated
  estimate for the free flow
  \begin{equation*}
    \textstyle
    \|\int_{0}^{+\infty}
    e^{i(t-t')\mathcal{D}}F(t')dt'\|_{L^{\infty}_{t}L^{2}}
    \lesssim
    \|\rho^{-1}|x|^{1/2}F\|_{L^{2}_{t}L^{2}}
  \end{equation*}
  can be proved by splitting the integral as
  \begin{equation*}
    \textstyle
    e^{it \mathcal{D}}
    \cdot
    \int_{0}^{+\infty}
        e^{-it'\mathcal{D}}F(t')dt'
  \end{equation*}
  and then using the conservation of $L^{2}$ norm for
  $e^{it \mathcal{D}}$ in combination with
  the dual of the smoothing estimate \eqref{eq:smooDiracV}
  in the case $V=0$. Then by a standard application of
  the Christ--Kiselev Lemma the same estimate holds for the
  truncated integral:
  \begin{equation*}
    \textstyle
    \|\int_{0}^{t}
    e^{i(t-t')\mathcal{D}}F(t')dt'\|_{L^{\infty}_{t}L^{2}}
    \lesssim
    \|\rho^{-1}|x|^{1/2}F\|_{L^{2}_{t}L^{2}}.
  \end{equation*}
  The corresponding estimate with angular regularity
  \begin{equation}\label{eq:angu}
    \textstyle
    \|\Lambda^{s}_{\omega}\int_{0}^{+\infty}
    e^{i(t-t')\mathcal{D}}F(t')dt'\|_{L^{\infty}_{t}L^{2}}
    \lesssim
    \|\rho^{-1}|x|^{1/2}\Lambda^{s}_{\omega}F\|_{L^{2}_{t}L^{2}}
  \end{equation}
  does not follow immediately since $\Lambda_{\omega}$ does
  not commute with $\mathcal{D}$; however, to overcome this
  difficulty, it is sufficient to replace $\Lambda_{\omega}^{s}$
  with the operator $|K|^{s}$ defined in \eqref{eq:defKs},
  which commutes with $\mathcal{D}$ and generates equivalent
  Sobolev norms on $\mathbb{S}^{2}$. 
  With the same arguments one proves
  \begin{equation*}
    \textstyle
    \|\Lambda^{s}_{\omega}\int_{0}^{+\infty}
    e^{i(t-t')\mathcal{D}}|D|F(t')dt'\|_{L^{\infty}_{t}L^{2}}
    \lesssim
    \|\rho^{-1}|x|^{1/2}\Lambda^{s}_{\omega}|D|F\|_{L^{2}_{t}L^{2}},
  \end{equation*}
  Thus we see that, using again the representation \eqref{eq:duha},
  the previous computations give also the second part of
  \eqref{eq:strichDVla} and the proof is concluded.
\end{proof}

\section{Global existence for small data}
\label{sec:GWPsmall}

We now prove Theorem \ref{the:GWPsmall}.
The proof is based on a straightforward
fixed point argument in the space
$X$ defined by the norm
\begin{equation}\label{eq:spaceX}
  \|u\|_{X}:=
  \|\Lambda^{s}_{\omega}u\|_{L^{2}_{t}L^{\infty}_{|x|}L^{2}_{\omega}}
  +
  \|\Lambda^{s}_{\omega}u\|_{L^{\infty}_{t}H^{1}_{x}}.
\end{equation}
Notice that estimate \eqref{eq:strichDVla0} can be written simply
\begin{equation}\label{eq:estX}
  \|e^{it(\mathcal{D}+V)}f\|_{X}\lesssim
  \|\Lambda^{s}_{\omega}\|_{H^{1}}.
\end{equation}
Define $u=\Phi(v)$ for $v\in X$ as the solution of the
linear problem
\begin{equation}\label{eq:linearized}
  iu_{t}+\mathcal{D}u+Vu=\langle \beta u,u \rangle \beta u,\qquad
  u(0,x)=u_{0}(x)
\end{equation}
and represent $u$ as
\begin{equation*}
  \textstyle
  u=\Phi(v)=
  e^{it(\mathcal{D}+V)}u_{0}-
    i \int_{0}^{t}e^{i(t-t')(\mathcal{D}+V)}
       \langle \beta u,u \rangle \beta u dt'.
\end{equation*}
Now by the product estimate \eqref{eq:est:prod} and
by \eqref{eq:estX} we have
\begin{equation*}
\begin{split}
  \|u\|_{X}
  &\lesssim
  \textstyle
  \|\Lambda^{s}_{\omega}f\|_{H^{1}}+
  \int_{0}^{\infty}
  \|e^{i(t-t')\mathcal{D}} \langle \beta u,u \rangle \beta u \|_{X}
  dt'
  \\
  & \lesssim
  \textstyle
  \|\Lambda^{s}_{\omega}f\|_{H^{1}}+
  \int_{0}^{\infty}
   \|\Lambda^{s}_{\omega}
    \langle \beta u,u \rangle \beta u\|_{H^{1}}dt'
  \equiv
  \|\Lambda^{s}_{\omega}f\|_{H^{1}}+
  \|\Lambda^{s}_{\omega}
    P(v,\overline{v})\|_{L^{1}H^{1}}.
\end{split}
\end{equation*}
Using again \eqref{eq:est:prod} we have
\begin{equation*}
  \|\Lambda_{\omega}^{s}(v^{3})\|_{L^{2}_{\omega}(\mathbb{S}^{2})}
  \lesssim
  \|\Lambda_{\omega}^{s}v\|^{3}_{L^{2}_{\omega}(\mathbb{S}^{2})}
\end{equation*}
so that
\begin{equation*}
  \|\Lambda_{\omega}^{s}(v^{3})\|_{L^{2}_{x}}
  \lesssim
  \|\Lambda_{\omega}^{s}v\|_{L^{2}_{x}}
  \|\Lambda_{\omega}^{s}v\|^{2}_{L^{\infty}_{|x|} L^{2}_{\omega}}
\end{equation*}
and
\begin{equation}\label{eq:prod1}
  \|\Lambda_{\omega}^{s}(v^{3})\|_{L^{1}_{t}L^{2}_{x}}
  \lesssim
  \|\Lambda_{\omega}^{s}v\|_{L^{\infty}_{t} L^{2}_{x}}
  \|\Lambda_{\omega}^{s}v\|^{2}
      _{L^{2}_{t}L^{\infty}_{|x|} L^{2}_{\omega}}
  \le \|v\|_{X}^{3}.
\end{equation}
In a similar way,
\begin{equation*}
  \|\Lambda_{\omega}^{s}\nabla(v^{3})\|
        _{L^{2}_{\omega}(\mathbb{S}^{2})}
  \lesssim
  \|\Lambda_{\omega}^{s}\nabla v\|
      _{L^{2}_{\omega}(\mathbb{S}^{2})}
  \|\Lambda_{\omega}^{s}v\|^{2}_{L^{2}_{\omega}(\mathbb{S}^{2})}
\end{equation*}
so that
\begin{equation*}
  \|\Lambda_{\omega}^{s}\nabla(v^{3})\|
        _{L^{2}_{x}}
  \lesssim
  \|\Lambda_{\omega}^{s}\nabla v\|
      _{L^{2}_{x}}
  \|\Lambda_{\omega}^{s}v\|^{2}
      _{L^{\infty}_{|x|} L^{2}_{\omega}}
\end{equation*}
and
\begin{equation}\label{eq:prod2}
  \|\Lambda_{\omega}^{s}\nabla(v^{3})\|_{L^{1}_{t}L^{2}_{x}}
  \lesssim
  \|\Lambda_{\omega}^{s}\nabla v\|_{L^{\infty}_{t} L^{2}_{x}}
  \|\Lambda_{\omega}^{s}v\|^{2}
      _{L^{2}_{t}L^{\infty}_{|x|} L^{2}_{\omega}}
  \le \|v\|_{X}^{3}.
\end{equation}
In conclusion, \eqref{eq:prod1} and \eqref{eq:prod2} imply
\begin{equation*}
   \|\Lambda^{s}_{\omega}
      P(v,\overline{v})\|_{L^{1}H^{1}}\lesssim
  \|v\|^{3}_{X}
\end{equation*}
and the estimate for $u=\Phi(v)$ is
\begin{equation*}
  \|u\|_{X}\equiv\|\Phi(v)\|_{X}\lesssim
  \|\Lambda^{s}_{\omega} f\|_{H^{1}}+\|v\|_{X}^{3}.
\end{equation*}
An analogous computation gives the estimate
\begin{equation*}
  \|\Phi(v)-\Phi(w)\|_{X}\lesssim
  \|v-w\|_{X}\cdot(\|v\|_{X}+\|w\|_{X})^{2}
\end{equation*}
and an application of the contraction mapping theorem
gives the existence and uniqueness of a global solution.
The proof of scattering is completely standard and is
omitted.

\section{Conserved quantities}
\label{sec:conserved}

We observe the conserved quantities for \eqref{eq:Dirac} (see \cite{ChaGla, OzaYam}).

\begin{lemma}\label{lem:conserved}
Let $u$ be a solution to \eqref{eq:Dirac}.
Then,
\[
\int_{\R^3} |u(t,x)|^2 dx, \quad
\int_{\R^3} (\gamma u(t,x), \overline{u(t,x)}) dx
\]
are independent of $t$.
\end{lemma}

\begin{proof}
Since $V(x)$ is hermitian, we have
\begin{align*}
& \frac{d}{dt} \int_{\R^3} |u(t,x)|^2 dx \\
& = \int_{\R^3} \left\{ i (\D u(t,x) +V(x) u(t,x) - (\beta u(t,x), u(t,x)) \beta u(t,x),u(t,x)) \right. \\
& \hspace*{50pt} \left. - i (u(t,x), \D u(t,x) + V(x) u(t,x) - (\beta u(t,x),u(t,x)) \beta u(t,x)) \right\} dx \\
& = \sum _{j=1}^3 \int_{\R^3} \partial_j (\alpha_j u(t,x),u(t,x)) dx =0.
\end{align*}
From
\[
\beta \gamma  =-\gamma  \beta, \quad
\overline{\alpha_j} \gamma  = \gamma  \alpha_j, \quad
\overline{V(x)} \gamma  = - \gamma  V(x),
\]
we have
\begin{align*}
& \frac{d}{dt} \int_{\R^3} (\gamma u(t,x), \overline{u(t,x)}) dx \\
& = \int_{\R^3} \left\{ i (\gamma  \D u(t,x) + \gamma V(x) u(t,x) - (\beta  u(t,x), u(t,x)) \gamma \beta u(t,x), \overline{u(t,x)}) \right. \\
& \hspace*{50pt} \left. +i(\gamma u(t,x), \overline{\D u(t,x)} + \overline{V(x) u(t,x)} - (\beta u(t,x),u(t,x)) \beta \overline{u(t,x)}) \right\} dx \\
& = \sum _{j=1}^3 \int_{\R^3} \partial_j (\gamma  \alpha_j u(t,x), \overline{u(t,x)}) dx =0.
\end{align*}
\end{proof}

From these conserved quantities, $(\beta u, u)=0$ for any $t \in \R$ provided that $\gamma  u_0 = \overline{u_0}$;
this is called the Lochak--Majorana condition \cite{Loc, Bac}.

\begin{corollary} \label{cor:vanish}
Let $u$ be a solution to \eqref{eq:Dirac} with $\gamma u_0 = \overline{u_0}$.
Then, $(\beta u, u)=0$ for any $t \in \R$.
\end{corollary}

\begin{proof}
From $| \gamma  u - \overline{u}|^2= 2 |u|^2 +2 \Re (\gamma u, \overline{u})$,
\[
\int_{\R^3} | \gamma  u(t,x) - \overline{u(t,x)}|^2 dx
\]
is also a conserved quantity.
By the assumption, $\gamma u= \overline{u}$ for any $t \in \R$.
Then,
\[
\langle \beta u,u \rangle 
= (\beta \gamma  \overline{u},u)
= -(\gamma  \beta \overline{u}, u)
= -(\beta \overline{u}, \gamma u)
= - \overline{\langle \beta u,u \rangle }.
\]
Since $\langle \beta u,u \rangle $ is real valued, we obtain $(\beta u, u)=0$.
\end{proof}

\section{Global existence for large data}
\label{sec:GWPlarge}

We now prove Theorem \ref{the:GWPlarge}.
Denote by $\chi_{0}=Pu_{0}$ the projection
of the initial data on the subspace $E$ 
(see \eqref{eq:defE}--\eqref{eq:defP}), and let
$\chi$ be a solution to
\begin{equation*}
i \partial _t \chi + \D \chi + V(x) \chi = ( \beta \chi, \chi) \beta \chi, \quad
\chi (0,x) = \chi_0(x).
\end{equation*}
From $A \chi_0 = \overline{\chi_0}$ and Corollary \ref{cor:vanish}, the nonlinear term vanishes.
In particular, $\chi$ is a solution to the linear problem
\[
i \partial _t \chi + \D \chi + V(x) \chi =0, \quad
\chi (0,x) = \chi_0(x)
\]
that is to say, $\chi=e^{it(\mathcal{D}+V)}\chi_{0}$.

Setting $v= u-\chi$, where $u$ is the solution to be
constructed,
we consider the following Cauchy problem:
\begin{equation} \label{eq:Dirac-v}
\begin{aligned}
& i \partial _t v + \D v + V(x) v = F(v, \chi), \\
& v (0,x) = v_0(x) := u_0(x)-\chi_0(x),
\end{aligned}
\end{equation}
where
\begin{align*}
F(v, \chi)
& := (\beta u, u) \beta u - (\beta \chi, \chi) \beta \chi \\
& = (\beta v, v) \beta v + (\beta \chi, v)v + (\beta v, \chi) v + (\beta v,v) \chi \\
& \quad + (\beta \chi, \chi) v + (\beta \chi, v) \chi + (\beta v, \chi) \chi.
\end{align*}
Let
\[
\| u \|_{X_I} := \| \Lambda^s_{\omega} u \|_{L^{\infty}_I H^1} + \| \Lambda^s_{\omega} u \|_{L^2_I L_{|x|}^{\infty} L_{\omega}^2}
\]
for an interval $I \subset \R$.
We define
\[
\Phi (v) (t) := e^{it(\D + V)} v_0 - i \int_0^t e^{i (t-t') (\D + V)} F(v(t'),\chi (t')) dt'.
\]
Since $\chi_0$ is not small, we shall divide the time interval into a finite number of subintervals such that the norm of $\chi$ is sufficiently small on each.

Let $C_0$ and $C_1$ be the absolute constants appearing in the estimates below.
From Theorem \ref{the:smooDV}, estimate \eqref{eq:strichDVla0},
there exists $T^{\ast}>0$ such that
\[
\| \Lambda^s_{\omega} \chi \|_{L_{[T^{\ast}, \infty)}^2L_{|x|}^{\infty} L_{\omega}^2} < \frac{1}{10(C_1 \| \Lambda^s_{\omega} \chi_0 \|_{H^1} +1)}.
\]
In addition, we can take $T_{\ast}>0$ satisfying
\[
\sup_{0 \le T \le T^{\ast}} \| \Lambda^s_{\omega} \chi \|_{L_{[T,T+T_{\ast}]}^2L_{|x|}^{\infty} L_{\omega}^2} < \frac{1}{10(C_1 \| \Lambda^s_{\omega} \chi_0 \|_{H^1} +1)}.
\]
Let $k$ be a minimum natural number satisfying $k T_{\ast} > T^{\ast}$.
We take sufficiently small $\eps >0$ with
\begin{equation} \label{eq:eps}
4 (2C_0)^{2(k+1)} C_1 \eps^2 + 2 (2C_0)^{k+1} C_1 (\| \Lambda^s_{\omega} \chi_0 \|_{H^1}+1) \eps < \frac{1}{10}.
\end{equation}

We assume that $\| \Lambda^s_{\omega} v_0 \|_{H^1} \le \eps$.
Again \eqref{eq:strichDVla0} yields
\[
\| \Phi (v) \|_{X_{[0,T]}}
\lesssim \| \Lambda^s_{\omega} v_0 \|_{H^1} + \| \Lambda^s_{\omega} F(v, \chi) \|_{L_{[0,T]}^{1} H^1}.
\]
For simplicity, we denote a cubic part with respect to $f$, $g$ and $h$ by $fgh$, e.g., $v^2 \chi$ means $(\beta \chi, v) \beta v$ or $(\beta v, \chi) \beta v$ or $(\beta v,v) \beta \chi$.
By \eqref{eq:est:prod}, we have
\[
\| \Lambda^s_{\omega} (v^3) \|_{L_{\omega}^2(\Si)} \lesssim \| \Lambda^s_{\omega} v \|_{L_{\omega}^2(\Si)}^3,
\]
\[
\| \Lambda^s_{\omega} (v^3) \|_{L_{[0,T]}^1L_x^2}
\lesssim \left\| \| \Lambda^s_{\omega} v \|_{L_x^2} \| \Lambda^s_{\omega} v\|_{L_{|x|}^{\infty} L_{\omega}^2}^2 \right\|_{L_{[0,T]}^1}
\lesssim \| \Lambda^s_{\omega} v \|_{L_{[0,T]}^{\infty} L_x^2} \| \Lambda^s_{\omega} v \|_{L_{[0,T]}^2L_{|x|}^{\infty} L_{\omega}^2}^2.
\]
Similarly, we have
\begin{align*}
\| \Lambda^s_{\omega} \nabla (v^3) \|_{L_{[0,T]}^1L_x^2}
& \lesssim \left\| \| \Lambda^s_{\omega} \nabla v \|_{L_x^2} \| \Lambda^s_{\omega} v\|_{L_{|x|}^{\infty} L_{\omega}^2}^2 \right\|_{L_{[0,T]}^1} \\
& \lesssim \| \Lambda^s_{\omega} \nabla v \|_{L_{[0,T]}^{\infty} L_x^2} \| \Lambda^s_{\omega} v \|_{L_{[0,T]}^2L_{|x|}^{\infty} L_{\omega}^2}^2.
\end{align*}
The calculation used above gives
\begin{align*}
& \| \Lambda^s_{\omega} (v^2 \chi) \|_{L_{[0,T]}^1L_x^2}
\lesssim \| \Lambda^s_{\omega} v \|_{L_{[0,T]}^{\infty} L_x^2} \| \Lambda^s_{\omega} v \|_{L_{[0,T]}^2L_{|x|}^{\infty} L_{\omega}^2} \| \Lambda^s_{\omega} \chi \|_{L_{[0,T]}^2L_{|x|}^{\infty} L_{\omega}^2},
\\
& \| \Lambda^s_{\omega} \nabla (v^2 \chi) \|_{L_{[0,T]}^1L_x^2}
\lesssim \| \Lambda^s_{\omega} \nabla \chi \|_{L_{[0,T]}^{\infty} L_x^2} \| \Lambda^s_{\omega} v \|_{L_{[0,T]}^2L_{|x|}^{\infty} L_{\omega}^2}^2 \\
& \hspace*{100pt} + \| \Lambda^s_{\omega} \nabla v \|_{L_{[0,T]}^{\infty} L_x^2} \| \Lambda^s_{\omega} \chi \|_{L_{[0,T]}^2L_{|x|}^{\infty} L_{\omega}^2} \| \Lambda^s_{\omega} v \|_{L_{[0,T]}^2L_{|x|}^{\infty} L_{\omega}^2},
\\
& \| \Lambda^s_{\omega} (v \chi^2) \|_{L_{[0,T]}^1L_x^2}
\lesssim \| \Lambda^s_{\omega} v \|_{L_{[0,T]}^{\infty} L_x^2} \| \Lambda^s_{\omega} \chi \|_{L_{[0,T]}^2L_{|x|}^{\infty} L_{\omega}^2}^2,
\\
& \| \Lambda^s_{\omega} \nabla (v \chi^2) \|_{L_{[0,T]}^1L_x^2}
\lesssim \| \Lambda^s_{\omega} \nabla \chi \|_{L_{[0,T]}^{\infty} L_x^2} \| \Lambda^s_{\omega} \chi \|_{L_{[0,T]}^2L_{|x|}^{\infty} L_{\omega}^2} \| \Lambda^s_{\omega} v \|_{L_{[0,T]}^2L_{|x|}^{\infty} L_{\omega}^2} \\
& \hspace*{100pt} + \| \Lambda^s_{\omega} \nabla v \|_{L_{[0,T]}^{\infty} L_x^2} \| \Lambda^s_{\omega} \chi \|_{L_{[0,T]}^2L_{|x|}^{\infty} L_{\omega}^2}^2.
\end{align*}

Hence, we have
\begin{align*}
\| \Phi (v) \|_{X_{[0,T_{\ast}]}}
& \le C_0 \| \Lambda^s_{\omega} v_0 \|_{H^1} + C_1 \| v \|_{X_{[0,T_{\ast}]}}^3 \\
& \quad + C_1 \left( \| \Lambda^s_{\omega} \chi_0 \|_{H^1} + \| \Lambda^s_{\omega} \chi \|_{L_{[0,T_{\ast}]}^2 L_{|x|}^{\infty} L_{\omega}^2} \right) \| v \|_{X_{[0,T_{\ast}]}}^2 \\ 
& \quad + C_1 \left( \| \Lambda^s_{\omega} \chi_0 \|_{H^1} + \| \Lambda^s_{\omega} \chi \|_{L_{[0,T_{\ast}]}^2L_{|x|}^{\infty} L_{\omega}^2} \right) \| \Lambda^s_{\omega} \chi \|_{L_{[0,T_{\ast}]}^2L_{|x|}^{\infty} L_{\omega}^2} \| v \|_{X_{[0,T_{\ast}]}} \\
& \le C_0 \| \Lambda^s_{\omega} v_0 \|_{H^1} + C_1 \| v \|_{X_{[0,T_{\ast}]}}^3 \\
& \quad + C_1 \left( \| \Lambda^s_{\omega} \chi_0 \|_{H^1} + \| \Lambda^s_{\omega} \chi \|_{L_{[0,T_{\ast}]}^2 L_{|x|}^{\infty} L_{\omega}^2} \right) \| v \|_{X_{[0,T_{\ast}]}}^2 + \frac{1}{10} \| v \|_{X_{[0,T_{\ast}]}}.
\end{align*}
Then, $\Phi$ is a mapping from $B_{1} := \{ v \in X_{[0,T_{\ast}]} : \| v \|_{X_{[0,T_{\ast}]}} \le 2 C_0 \eps \}$ into itself because of \eqref{eq:eps}.

Let $v_1$ and $v_2$ be solutions to \eqref{eq:Dirac-v}.
The difference $v_1-v_2$ satisfies
\[
i \partial _t (v_1-v_2) + \D (v_1-v_2) + V(x) \beta (v_1-v_2) = F(v_1, \chi) - F(v_2,\chi), \quad
(v_1-v_2) (0,x) = 0.
\]
Accordingly, for $v_1, v_2 \in B_{1}$, we have
\begin{align*}
& \| \Phi (v_1) - \Phi (v_2) \|_{X_{[0,T_{\ast}]}} \\
& \le C_1 \Big( (\| v_1 \|_{X_{[0,T_{\ast}]}}+ \| v_2 \|_{X_{[0,T_{\ast}]}}) (\| \Lambda^s_{\omega} v_1 \|_{L_{[0,T_{\ast}]}^2L_{|x|}^{\infty} L_{\omega}^2} + \| \Lambda^s_{\omega} v_2 \|_{L_{[0,T_{\ast}]}^2L_{|x|}^{\infty} L_{\omega}^2}) \\
& \quad + (\| v_1 \|_{X_{[0,T_{\ast}]}}+ \| v_2 \|_{X_{[0,T_{\ast}]}}) ( \| \Lambda^s_{\omega} \chi_0 \|_{H^1} + \| \Lambda^s_{\omega} \chi \|_{L_{[0,T_{\ast}]}^2L_{|x|}^{\infty} L_{\omega}^2} ) \\
& \quad + ( \| \Lambda^s_{\omega} \chi_0 \|_{H^1} + \| \Lambda^s_{\omega} \chi \|_{L_{[0,T_{\ast}]}^2L_{|x|}^{\infty} L_{\omega}^2} ) \| \Lambda^s_{\omega} \chi \|_{L_{[0,T_{\ast}]}^2L_{|x|}^{\infty} L_{\omega}^2} \Big) \| v_1-v_2 \|_{X_{[0,T_{\ast}]}} \\
& \le \frac{1}{2} \| v_1-v_2 \|_{X_{[0,T_{\ast}]}}.
\end{align*}
Therefore, $\Phi : B_1 \mapsto B_1$ is a contraction mapping, and we obtain a unique solution $v$ to \eqref{eq:Dirac-v}.
Since the existence time $T_{\ast}$ depends only on $\chi _0$, we can extend the existence time to $[0, kT_{\ast}]$.
Indeed, setting $I_n := [(n-1)T_{\ast}, n T_{\ast}]$ for $n =1,2, \dots , k$, we have
\begin{align*}
\| \Phi (v) \|_{X_{I_n}}
& \le C_0 \| \Lambda^s_{\omega} v ((n-1)T_{\ast}) \|_{H^1} + C_1 \| v \|_{X_{I_n}}^3 \\
& \quad + C_1 \left( \| \Lambda^s_{\omega} \chi_0 \|_{H^1} + \| \Lambda^s_{\omega} \chi \|_{L_{I_n}^2L_{|x|}^{\infty} L_{\omega}^2} \right) \| v \|_{X_{I_n}}^2 \\ 
& \quad + C_1 \left( \| \Lambda^s_{\omega} \chi_0 \|_{H^1} + \| \Lambda^s_{\omega} \chi \|_{L_{I_n}^2L_{|x|}^{\infty} L_{\omega}^2} \right) \| \Lambda^s_{\omega} \chi \|_{L_{I_n}^2L_{|x|}^{\infty} L_{\omega}^2} \| v \|_{X_{I_n}} \\
& \le C_0 \| \Lambda^s_{\omega} v ((n-1)T_{\ast}) \|_{H^1} + C_1 \| v \|_{X_{I_n}}^3 \\
& \quad + C_1 \left( \| \Lambda^s_{\omega} \chi_0 \|_{H^1} + \| \Lambda^s_{\omega} \chi \|_{L_{I_n}^2L_{|x|}^{\infty} L_{\omega}^2} \right) \| v \|_{X_{I_n}}^2 + \frac{1}{10} \| v \|_{X_{I_n}}.
\end{align*}
Then, $\Phi$ is a mapping from $B_n := \{ v \in X_{I_n} : \| v \|_{X_{I_n}} \le (2C_0)^n \eps \}$ into itself because of \eqref{eq:eps}.
The estimate for the difference is similarly handled.
Hence, the existence of a unique solution $v$ to \eqref{eq:Dirac-v} follows from the contraction mapping theorem.
Thus, we obtain the unique solution $u$ on the time interval $[0,kT_{\ast}]$.
Similarly, we have
\begin{align*}
& \| \Phi (v) \|_{X_{[T^{\ast},\infty)}} \\
& \le C_0 \| \Lambda^s_{\omega} v(T^{\ast}) \|_{H^1} + C_1 \| v \|_{X_{[T^{\ast},\infty)}}^3 \\
& \quad + C_1 \left( \| \Lambda^s_{\omega} \chi_0 \|_{H^1} + \| \Lambda^s_{\omega} \chi \|_{L_{[T^{\ast},\infty)}^2 L_{|x|}^{\infty} L_{\omega}^2} \right) \| v \|_{X_{[T^{\ast},\infty)}}^2 \\ 
& \quad + C_1 \left( \| \Lambda^s_{\omega} \chi_0 \|_{H^1} + \| \Lambda^s_{\omega} \chi \|_{L_{[T^{\ast},\infty)}^2L_{|x|}^{\infty} L_{\omega}^2} \right) \| \Lambda^s_{\omega} \chi \|_{L_{[T^{\ast},\infty)}^2L_{|x|}^{\infty} L_{\omega}^2} \| v \|_{X_{[T^{\ast},\infty)}} \\
& \le C_0 \| \Lambda^s_{\omega} v(T^{\ast}) \|_{H^1} + C_1 \| v \|_{X_{[T^{\ast},\infty)}}^3 \\
& \quad + C_1 \left( \| \Lambda^s_{\omega} \chi_0 \|_{H^1} + \| \Lambda^s_{\omega} \chi \|_{L_{[T^{\ast},\infty)}^2 L_{|x|}^{\infty} L_{\omega}^2} \right) \| v \|_{X_{[T^{\ast},\infty)}}^2 + \frac{1}{10} \| v \|_{X_{[T^{\ast},\infty)}}.
\end{align*}
The estimate for the difference follows in the same manner.
Then, $\Phi$ is a contraction mapping from $B_{\infty} := \{ v \in X_{[T^{\ast},\infty)} : \| v \|_{X_{[T^{\ast},\infty)}} \le (2 C_0)^{k+1} \eps \}$ into itself because of \eqref{eq:eps}.

To show the scattering, we set
\[
v_+ := v_0 - i \int_0^{\infty} e^{-it'(\D + V)} F(v(t'),\chi (t')) dt',
\]
which satisfies $\Lambda^s_{\omega} v_+ \in H^1(\R^3)$ because $\Lambda^s_{\omega} F(v,\chi) \in L^1(\R; H^1(\R^3))$.
Then,
\begin{align*}
& \| \Lambda^s_{\omega} v(t) - \Lambda^s_{\omega} e^{it(\D + V)} v_+ \|_{H^1} \\
& \lesssim \left\| \Lambda^s_{\omega} \int_t^{\infty} e^{i(t-t')(\D + V)} F(v(t'), \chi (t')) dt' \right\|_{H^1} \\
& \lesssim \| \Lambda^s_{\omega} F(v, \chi ) \|_{L^1_{[t,\infty)}H^1} \\
& \lesssim \| \Lambda^s_{\omega} v \|_{L^2_{[t,\infty)} L_{|x|}^{\infty} L_{\omega}^2}^2 \| \Lambda^s_{\omega} v \|_{L^{\infty}_{[t,\infty)} H^1} \\
& \quad + \left( \| \Lambda^s_{\omega} \chi_0 \|_{H^1} + \| \Lambda^s_{\omega} \chi \|_{L_{[t,\infty)}^2 L_{|x|}^{\infty} L_{\omega}^2} \right) \| \Lambda^s_{\omega} v \|_{L^2_{[t,\infty)} L_{|x|}^{\infty} L_{\omega}^2} \| \Lambda^s_{\omega} v \|_{L^{\infty}_{[t,\infty)} H^1} \\ 
& \quad + \left( \| \Lambda^s_{\omega} \chi_0 \|_{H^1} + \| \Lambda^s_{\omega} \chi \|_{L_{[t,\infty)}^2L_{|x|}^{\infty} L_{\omega}^2} \right) \| \Lambda^s_{\omega} \chi \|_{L_{[t,\infty)}^2L_{|x|}^{\infty} L_{\omega}^2} \| \Lambda^s_{\omega} v \|_{L^{\infty}_{[t,\infty)} H^1}
\end{align*}
for any $t>0$.
Therefore,
\[
\lim_{t \rightarrow \infty} \| \Lambda^s_{\omega} v(t) - \Lambda^s_{\omega} e^{it(\D + V)} v_+ \|_{H^1} =0.
\]
From $\chi (t)= e^{it(\D + V)} \chi_0$, setting $u_+ := v_+ + \chi_0$, we obtain the desired result.

\mbox{}

\section*{Acknowledgment}

The work of the second author was supported by JSPS KAKENHI Grant number JP 16K17624.

% b_f_post
%%% >>> BIBLATEX:
% \printbibliography
%%% >>> BIBTEX: (cancellare \usepackage{biblatex})
% \bibliography{/Users/piero/.bib/bibliodatabase.bib}
% \bibliographystyle{abbrv}
%%% >>> entrambe si possono usare contemporaneamente a:
% \begin{\thebibliography}{10}  DATI  \end{thebiblography}

\end{document}